\documentclass[10pt]{article}
\usepackage[final]{epsfig}
\usepackage{amsfonts, amsmath, amssymb, amsthm}
\usepackage{latexsym}
\usepackage{color}

\numberwithin{equation}{section}

\newtheorem{theorem}{Theorem}[section]
\newtheorem{lemma}[theorem]{Lemma}
\newtheorem{remark}[theorem]{Remark}
\newtheorem{example}[theorem]{Example}
\newtheorem{proposition}[theorem]{Proposition}
\newtheorem{defn}[theorem]{Definition}
\newtheorem{corollary}[theorem]{Corollary}

\newcommand{\R}{\mathbb{R}}

\newcommand{\C}{\mathbb{C}}
\newcommand{\Z}{\mathbb{Z}}
\newcommand{\N}{\mathbb{N}}

\newcommand{\be}{\begin{equation}}
\newcommand{\ee}{\end{equation}}

\newcommand{\gammatildeepsilonbeta}{\tilde{\gamma}_{\epsilon,\beta}}
\newcommand{\LcapL}{L^1\cap L^3}
\newcommand{\cala}{{\mathcal{A}}}
\newcommand{\calR}{{\mathcal{R}}}
\newcommand{\calE}{{\mathcal{E}}}
\newcommand{\calM}{{\mathcal{M}}}

\newcommand{\bthe}{\begin{theorem}}
\newcommand{\ethe}{\end{theorem}}

\newcommand{\ben}{\begin{enumerate}}
\newcommand{\een}{\end{enumerate}}

\newcommand{\beq}{\begin{equation}}
\newcommand{\eeq}{\end{equation}}

\newcommand{\ble}{\begin{lemma}}
\newcommand{\ele}{\end{lemma}}

\newcommand{\bde}{\begin{definition}}
\newcommand{\ede}{\end{definition}}

\newcommand{\bco}{\begin{corollary}}
\newcommand{\eco}{\end{corollary}}

\newcommand{\bpr}{\begin{proposition}}
\newcommand{\epr}{\end{proposition}}

\newcommand{\bexam}{\begin{example}\rm}
\newcommand{\eexam}{\halmos\end{example}}

\newcommand{\beao}{\begin{eqnarray*}}
\newcommand{\eeao}{\end{eqnarray*}\noindent}

\newcommand{\beam}{\begin{eqnarray}}
\newcommand{\eeam}{\end{eqnarray}\noindent}

\newcommand{\barr}{\begin{array}}
\newcommand{\earr}{\end{array}}

\newcommand{\bproof}{\begin{proof}}
\newcommand{\eproof}{\end{proof}}

\newcommand{\la}{{\lambda}}
\newcommand{\ga}{{\gamma}}
\newcommand{\Ga}{{\Gamma}}

\newcommand{\eps}{\varepsilon}

\def\bbr{{\Bbb R}}
\def\bbn{{\Bbb N}}

\def\calp{{\mathcal{P}}}

\def\calt{{\mathcal{T}}}
\def\calR{{\mathcal{R}}}

\newcommand{\halmos}{\quad\hfill\mbox{$\Box$}}

\newcommand{\EEE}{\color{black}}

\begin{document}

\begin{center}
{\Huge Density functional theory and\\
 optimal transportation with Coulomb cost}  \\

\normalsize
\vspace{0.4in}

Codina Cotar, Gero Friesecke and Claudia Kl\"uppelberg  \\
Department of Mathematics  \\
Technische Universit\"at M\"unchen  \\
cotar@ma.tum.de, gf@ma.tum.de, cklu@ma.tum.de \\[2mm]
\end{center}

\noindent
\normalsize
{\bf Abstract.} We present here novel insight into exchange-correlation functionals in density functional theory,
based on the viewpoint of optimal transport. We show that in the case of two electrons and in the semiclassical limit,
the exact exchange-correlation functional reduces to a very interesting functional of novel form, which depends on an optimal
transport map $T$ associated with a given density $\rho$. 
Since the above limit is strongly correlated, the limit functional
yields insight into electron correlations.
We prove the existence and uniqueness of such an optimal map for any
number of electrons and each $\rho$, and determine the map explicitly in the case when $\rho$ is radially symmetric.

\vspace{0.2in}



\tableofcontents

\vspace{0.2in}

\section{Introduction}
The precise modelling of electron correlations continues to constitute
the major obstacle in developing high-accuracy, low-cost methods for electronic structure
computations in molecules and solids. In this article we shed new light on the longstanding problem
of how to accurately incorporate electron correlation into density functional theory (DFT),
by deriving and analyzing the semiclassical limit of the exact Hohenberg-Kohn
functional with the single-particle density $\rho$ held fixed. In this limit, we find that the exact
functional reduces formally to a very interesting functional of novel form which
depends on an {\it optimal transport map} $T$ associated with a given density $\rho$. Our work thereby links
DFT, which is a large and very active research area in physics and chemistry
[PY95, FNM03, Ra09], for the first time to optimal transportation theory, which has
recently become a very active area in mathematics \cite{GM96, Ru96, Villani}.

In optimal transportation theory the goal is to transport ``mass'' from an initial density $\rho_A$
to a target density $\rho_B$ in such a way that the ``cost'' $c(x,y)$ for transporting mass from $x$ to $y$
is minimized. Mathematically, this means that one minimizes a cost functional $\int c(x,y) d\gamma(x,y)$ over
a set of joint measures $\gamma$ (in physics terminology: pair densities) subject to fixed marginals (single-particle densities).
See below for a precise formulation. The main mathematical novelty of the optimal transportation
problem arising from DFT,
\beq \label{firsteq}
   \mbox{Minimize}\; \int_{\R^6} \frac{1}{|x-y|} d\gamma(x,y) \mbox{ subject to equal marginals }\rho,
\eeq
is that the cost, which is given by the Coulomb law
$c(x,y)={1}/{|x-y|}$, {\it decreases} rather than increases with distance and has a {\it singularity} on
the diagonal.

Our goals in this paper are
\\[1mm]
(i) to prove that for any given single-particle density $\rho$, the optimal transportation problem with Coulomb cost possesses
a unique minimizer which is given by an optimal transport map $T_\rho$ associated with $\rho$. (It is well known
that uniqueness is {\it false} for the seemingly simpler cost function $c(x,y)=|x-y|$.)
\\[1mm]
(ii) to derive an explicit formula for the optimal map in the case when $\rho$ is radially symmetric. (Note that in
physics, radial densities arise as atomic ground state densities for many elements such as He, Li, N, Ne,
Na, Mg, Cu.)
\\[1mm]
(iii) to prove that DFT with electron interaction energy given by
the optimal transportation cost $E_{OT}[\rho]$
(defined as the minimum cost in (\ref{firsteq}))
is the semiclassical limit of exact Hohenberg-Kohn DFT in case of two electrons, and establish basic properties
such as that it is a rigorous lower bound to exact DFT for any number of electrons.
\\[1mm]
We do not know whether our semiclassical limit result remains true for a general number of electrons.
As explained in Section 5, this question is related to the representability
problem for two-particle density matrices [CY02].
\\[1mm]
To prove (i) and (ii) we adapt geometric methods as developed in~[GM96], \cite{Ru96}, and~\cite{SK92}, for cost functions that
increase with distance. In our case of decreasing cost functions, one can still geometrically construct a potential
which specifies both in which direction and how far to move the mass of $\rho_A$ which
sits near $x$. To prove (iii) we will need to make modifications to the optimal
transport plan which yields $E_{OT}[\rho]$, since any wave function whose pair density is given by the optimal plan
has infinite kinetic energy. The main technical idea here is a construction to re-instate the original marginals
after smoothing.

The optimal transportation functional $E_{OT}[\rho]$ which emerges as a limit of the Hohenberg-Kohn functional
should be viewed as a natural ``opposite'' of the well known mean field functional
$J[\rho]=\frac12\int_{\R^6} \frac{\rho(x)\, \rho(y)}{|x-y|} dx\, dy$:
it arises in a strongly correlated rather than a de-correlated limit, thereby yielding valuable qualitative
insight into electron correlations. We also believe that $E_{OT}[\rho]$ has a role to play in the design of quantitative
competitors to existing exchange-correlation functionals: it provides an alternative starting point
of novel functional form for designing approximations, and -- just like the mean field functional -- could be incorporated
as an ingredient into hybrid functionals. Basic quantitative issues are addressed in the companion paper [CFK11].

This paper is structured as follows. In Section 2 we discuss density functional theory from a mathematical perspective.
In section~$3$ we introduce optimal transport theory, prove in Theorem~\ref{uniqgeom} the main result of the section,
the uniqueness of the optimal transport map, and establish some of its general properties.
In section~$4$ we give in Theorem \ref{bigdim} an explicit formula for the optimal
map for equal, radially symmetric marginals. In section~5 we compare the optimal transportation cost
$E_{OT}[\rho]$ to the exact Hohenberg-Kohn functional, and show
that it is its semiclassical limit in the case of two electrons (Theorem \ref{limit}), as well as a lower bound
for any number of electrons (Theorem \ref{lowbound}).

\section{Density functional theory}

Density functional theory (DFT) was introduced by Hohenberg, Kohn and Sham in the 1960s in two fundamental
papers [HK64, KS65], as an approximate computational method for solving the many-electron Schr\"odinger
equation whose computational cost remains feasible even for large systems. In this theory,
one only computes the single-particle density instead of the full wave function. In order
to obtain a closed equation, a closure assumption is needed which expresses the pair density in
terms of the single-particle density. A simple ``independence'' ansatz has turned out to be far
too crude in practice. A huge effort has gone into developing
corrections which account for the failure of independence [PY95, FNM03, Ra09].

Our goal in this section is to discuss DFT from a mathematical perspective.

The key quantity DFT aims to predict is the ground state energy $E_0$ of a molecule as
a function of the atomic positions. From this, further properties can be readily extracted,
for instance, in order to determine a molecule's stable equilibrium shapes one minimizes
$E_0$ (locally or globally) over atomic positions.

Starting point for developing DFT models is the ``exact'' (non-relativistic, Born-Oppenheimer)
quantum mechanical ground state energy $E_0^{QM}$. The definition contains some details which may look a bit complicated
to readers not familiar with many-particle quantum theory,
but the basic mathematical structure relevant to developing DFT models is simple. $E_0^{QM}$ is the minimum value of a suitable
energy functional ${\cal E}$ over a suitable class of functions ${\cal A}$.
\subsection{Exact ground state energy}
The detailed definition is as follows.
Consider a molecule with atomic nuclei at positions $R_1,..,R_M\in\R^3$, with charges $Z_1,..,Z_M\in\N$ ($Z=1$ means
hydrogen, $Z=2$ Helium, $Z=3$ Lithium, $Z=4$ Beryllium, $Z=5$ Boron, $Z=6$ Carbon, $Z=7$ Nitrogen,
$Z=8$ Oxygen, and so on), and with $N$ electrons. The energy functional depends on the positions and charges of
the nuclei only through the ensuing Coulomb potential
\be \label{potential}
   v(x) = - \sum_{\alpha=1}^M \frac{Z_\alpha}{|x-R_\alpha|},\quad x\in\R^3.
\ee
For the discussion below, the potential $v$ can more generally be any function in the space $L^{3/2}(\R^3)+L^\infty(\R^3) = \{v_1+v_2\, | \, v_1\in L^{3/2}(\R^3),
\, v_2\in L^\infty(\R^3)\}$. 
The class of functions ${\cal A}$ is given by
\be \label{class}
  {\cal A} = \{ \Psi \in L^2((\R^3\times\Z_2)^N;\C) \, | \, \nabla\Psi \in L^2, \, \Psi \, \mbox{antisymmetric}, \,
                ||\Psi||_{L^2}=1\}.
\ee
(Here antisymmetric means $\Psi(z_{\sigma(1)}, .., z_{\sigma(N)}) = sgn(\sigma)\Psi(z_1,..,z_N)$ for all permutations $\sigma$,
where $z_1,...,z_N\in\R^3\times\Z_2$ are space-spin-coordinates for the $N$ electrons. Spin will not play a big
role in the sequel, but the fact that the functions $\Psi$ depend on all the positions of all the electrons
is important. It leads to the fact that we will have to deal with ``N-point distributions''.) Elements of ${\cal A}$ are
called (N-electron) {\it wave functions}.

The energy functional is given by
\be \label{EQM}
  {\cal E}^{QM}[\Psi] = T[\Psi] + V_{ne}[\Psi] + V_{ee}[\Psi]
\ee
where (employing the notation $z_i=(x_i,s_i)\in\R^3\times\Z_2$, $\int dz_i = \sum_{s_i\in\Z_2}\int_{\R^3}dx_i$)
$$
  T[\Psi] = \frac12 \int ... \int \sum_{i=1}^N |\nabla_{x_i}\Psi(x_1,s_1,..,x_N,s_N)|^2 dz_1..dz_N
$$
is the kinetic energy,
$$
  V_{ne}[\Psi] = \int ... \int \sum_{i=1}^N v(x_i)|\Psi(x_1,s_1,..,x_N,s_N)|^2 dz_1..dz_N
$$
is the electron-nuclei interaction energy, and
$$
  V_{ee}[\Psi] = \int ... \int \sum_{1\le i < j \le N} \frac{1}{|x_i-x_j|} |\Psi(x_1,s_1,..,x_N,s_N)|^2 dz_1..dz_N
$$
is the electron-electron interaction energy.

The quantum mechanical ground state energy is defined as
\be \label{VP}
  E_0^{QM} = \inf_{\Psi\in{\mathcal A}} {\cal E}^{QM}[\Psi].
\ee
In the usual case where $v$ is given by (\ref{potential}) and $N=\sum_{\alpha=1}^MZ_\alpha$ (neutral molecules), it is a basic theorem due to Zhislin that the infimum is
attained. For a simple proof see [Fr03].

Since the energy functional is a quadratic form, we could equivalently have defined $E_0^{QM}$ as the lowest eigenvalue of
the corresponding linear partial differential operator $-\frac12\sum_{i=1}^N\Delta_{x_i} + \sum_{i=1}^N v(x_i) +
\sum_{i<j}\frac{1}{|x_i-x_j|}$. This formulation is useful for many other purposes, but -- unlike (\ref{VP}) --
does not play an important role in DFT.

\subsection{Probabilistic interpretation; marginals}
The absolute value squared of $\Psi$ can be interpreted as an $N$-point probability distribution,
$$
  |\Psi(x_1,s_1,..,x_N,s_N)|^2 = \mbox{ probability density that the electrons are at positions $x_i$ with spins
  $s_i$}.
$$
In quantum mechanics this is known as the Born interpretation. Note that the above function is nonnegative and
integrates to 1, due to the requirement that $\Psi$ has $L^2$ norm 1. (In fact this was the physical motivation
for this requirement.)

Various partial marginals will play an important role. First, by integrating out the spins we obtain the
$N$-point position density:
\be \label{rhoN}
  \rho_N^\Psi(x_1,..,x_N) = \sum_{s_1,..,s_N\in\Z_2} |\Psi(x_1,s_1,..,x_N,s_N)|^2.
\ee
Next, by integrating out all but two respectively one electron positions we obtain the pair density and the
single particle density:\begin{align}
  \rho_2^{\Psi}(x_1,x_2) &= {N\choose 2}\int_{\R^{3(N-2)}} \rho_N^\Psi(x_1,..,x_N)dx_3..dx_N, \label{rho2} \\
  \rho^{\Psi}(x_1) &= N\int_{\R^{3(N-1)}} \rho_N^\Psi(x_1,..,x_N)dx_2..dx_N. \label{rho}
\end{align}
For the rest of this section, we drop the superscript $\Psi$ from $\rho^\Psi$ and $\rho_2^\Psi$.

The normalization factors are a convention in quantum mechanics so that $\rho$ integrates to the number
of particles and $\rho_2$ to the number of pairs in the system. (In Section 3 we find it convenient to work with
the corresponding probability densities, normalized so as to integrate to 1.)
With the above conventions, the important fact that $\rho$ is a marginal distribution of $\rho_2$ takes the form
\be \label{marginal}
  \frac{N}{{N\choose 2}} \int \rho_2(x,y) \, dy = \rho(x), \;\; \frac{N}{{N\choose 2}} \int \rho_2(x,y) \, dx = \rho(y).
\ee

The relevance of $\rho$ and $\rho_2$ for determining the ground state energy (\ref{VP})
come from the fact that the electron-nuclei energy $V_{ne}$ and
the electron-electron energy $V_{ee}$ in (\ref{EQM}) depend only on these.
\begin{lemma} \label{DFT1} With the above definitions, for any $\Psi\in{\cal A}$ we have
\beam\label{vnevee}
   V_{ne}[\Psi] = \int_{\R^3} v(x)\rho(x)\, dx, \;\;\; V_{ee}[\Psi] = \int_{\R^6}
   \frac{1}{|x-y|}\rho_2(x,y)\, dx\, dy.
\eeam
\end{lemma}
\bproof
This follows from definitions (\ref{rho2}), (\ref{rho}) and the fact that due to the antisymmetry of $\Psi$,
$|\Psi(z_1,..,z_N)|^2$ is a symmetric function of the $z_i$.
\eproof
In the sequel we write $V_{ne}[\rho]$, $V_{ee}[\rho_2]$ instead of $V_{ne}[\Psi]$, $V_{ee}[\Psi]$.

Also, we note that the
space of densities arising from functions $\Psi\in\cala$,
\begin{equation} \label{admdens}
     \calR := \{ \rho \, : \, \R^3\to\R \, | \, \rho \mbox{ is the density (\ref{rho}) 
     of some }\Psi\in\cala\},
\end{equation}
is known explicitly: by a result of Lieb [Li83],
\beq \label{admdens2}
   {\cal R} = \{ \rho \, : \, \R^3 \to \R \, | \, \rho\ge 0, \, \sqrt{\rho}\in H^1(\R^3), \, \int_{\R^3}\rho(x)\, dx = N\},
\eeq
where $H^1(\R^3)$ is the usual Sobolev space $\{ u \in L^2(\R^3) \, | \, \nabla u\in L^2(\R^3)\}$.
\subsection{Universal Hohenberg-Kohn functional}
The expression for $V_{ne}$ derived in Lemma \ref{DFT1} leads to the following well known partitioning of the minimization (\ref{VP})
into a double minimization (first minimize over $\Psi$ subject to fixed $\rho$, then over $\rho$):
\beq \label{HKVP}
   E_0^{QM} = \inf_{\rho\in\calR} \Big\{ F_{HK}[\rho] + \int_{\R^3} v(x)\,\rho(x) dx \Big\}
\eeq
with
\beq \label{FHK}
   F_{HK}[\rho] := \inf_{\Psi\in\cala, \, \Psi\mapsto\rho} \Big\{T[\Psi]+V_{ee}[\rho_2]\Big\}.
\eeq
Here and below, the notation $\Psi\mapsto\rho$ means that $\Psi$ has single-particle density $\rho$.
Note that $F_{HK}$ is a universal functional of $\rho$, in the sense that it does not depend on the external potential $v$,
and is called the {\it Hohenberg-Kohn functional}. It is defined on the admissible set (\ref{admdens}).
The above constrained-search definition of $F_{HK}$ is due to Levy and
Lieb [Le79, Li83]; in the original Hohenberg-Kohn paper [HK64] the functional was constructed in a more indirect and slightly less general
way, requiring that $\rho$ be the density of some $\Psi$ which is a non-degenerate ground state
of $\calE^{QM}$ for some potential $v$.

\subsection{Exchange-correlation functionals}

The problem with definition (\ref{VP}) of $E_0^{QM}$, as well as definition (\ref{FHK}) of the `exact' density
functional $F_{HK}$, is that it is unfeasible in practice except when
the number $N$ of particles is very small. This is due to the so-called
problem of exponential scaling: the functions over which one minimizes are functions on $\R^{3N}$ and the discretization of $\R^{3N}$ requires a $K^N$-point grid if the single-particle
space $\R^3$ is discretized by a $K$-point grid.

This problem would disappear if we could accurately approximate
$V_{ee}$ in the variational principle (\ref{VP}) by a functional $\tilde{V}_{ee}$ of
$\rho$ instead of $\rho_2$,
\be \label{keyapprox}
   V_{ee}[\rho_2] \approx \tilde{V}_{ee}[\rho].
\ee
(Why this is so is not completely trivial, since there remains $T$ to deal with, but see eq. (\ref{KE2}) below.)
Thus in DFT one approximates the variational principle for the ground state energy $E_0^{QM}$ by:
\be \label{DFTen}
   E_0^{DFT} = \inf_{\Psi\in {\cal A}} \Big\{ T[\Psi] + V_{ne}[\rho] + \tilde{V}_{ee}[\rho]\Big\}.
\ee
Physically, this means (in the light of Lemma \ref{DFT1}) that in DFT, interactions of electrons with an
external environment, such as the Coulomb forces exerted by an array of atomic nuclei, are
included exactly, but electron-electron interactions have to be suitably ``modelled''. 
By partitioning the minimization in (\ref{DFTen}) analogously to (\ref{HKVP}), (\ref{FHK}), $E_0^{DFT}$ can be obtained
by minimization of a functional of $\rho$ alone,
\be \label{DFTen2} 
   E_0^{DFT} = \inf_{\rho\in {\cal R}} \Big\{ T_{QM}[\rho] + V_{ne}[\rho] + \tilde{V}_{ee}[\rho]\Big\}, \;\;\; 
   T_{QM}[\rho] := \inf_{\Psi\in\cala, \, \Psi\mapsto\rho} T[\Psi].
\ee
The minimization over the ``large'' space ${\cal A}$ of functions on $(\R^3\times\Z_2)^N$ in
(\ref{DFTen2}) can now be replaced by a minimization over a much ``smaller'' space. As can be shown with the
help of reduced density matrices [CY02], and abbreviating $\int=\sum_{s\in\Z_2}\int_{\R^3}$,
\begin{eqnarray}
\label{KE2}
   \inf_{\Psi\in\cala, \, \Psi\mapsto\rho} T[\Psi] = \inf \Big\{ \sum_{i=1}^\infty \frac{\lambda_i}{2}
   \int|\nabla\phi_i|^2 \, & | & \, 0\le\lambda_i\le 1, \, \sum_{i=1}^\infty\lambda_i=N, \,\nonumber \\
   & &  \phi_i\in H^1(R^3\times\Z_2), \, \int\phi_i\overline{\phi_j} = \delta_{ij},
   \, \sum_{i=1}^\infty \sum_{s\in\Z_2}|\phi_i(x,s)|^2 = \rho(x) \Big\}.\nonumber\\
\end{eqnarray}
After truncating the sum after an appropriate number $i_{max}$ of terms (the standard truncation
being $i_{max}=N$, yielding the Kohn-Sham kinetic energy functional [KS65])
and discretizing $\R^3$ by a K-point grid, the number of degrees of freedom of the right hand side
scales linearly instead of exponentially in $N$.

By means of this fact, the task of eliminating the exponential complexity of (\ref{VP}) is reduced to the following
\\[2mm]
{\bf Fundamental problem of DFT} {\it Design accurate approximations of the form (\ref{keyapprox}). In other words,
approximate a simple explicit functional of the pair density $\rho_2$, a function on $\R^6$, by
a functional of its joint right and left marginal $\rho$, a function on $\R^3$.
Note that the approximations only need to be accurate for single-particle densities and pair densities of ground states of molecules,
not arbitrary states. Elsewhere it suffices that the approximations give a reasonably good lower bound so as to avoid spurious minimizers.}

\bexam (statistical independence) The simplest idea would be to assume statistical independence,
\be \label{indansatz}
    \rho_2(x,y) \approx \frac{1}{2}\rho(x)\rho(y)
\ee
(the factor $1/2$ coming from the normalization factors in (\ref{rho}), (\ref{rho2})),
and substitute this ansatz into the formula for $V_{ee}$ derived in Lemma \ref{DFT1}. This leads to taking
\be \label{indansatzforVee}
   \tilde{V}_{ee}[\rho] = \frac12 \int_{\R^6}\frac{1}{|x-y|} \rho(x)\rho(y)\, dx\, dy =: J[\rho],
\ee
i.e. $V_{ee}$ is replaced by the Coulomb self-repulsion of the single-particle density. The above mean field functional
appears, for instance, in Thomas-Fermi-theory.

In modern DFT, the very naive ansatz (\ref{indansatz}) was never used, but -- without this being natural from a probabilistic point of view --
the convention is to include corrections to it additively, i.e. one makes an ansatz
\be \label{DFTansatz}
   \tilde{V}_{ee}[\rho] = J[\rho] + E_{xc}[\rho],
\ee
the additive correction being called an {\it exchange-correlation functional}. This notational convention should
not, of course, prevent us from contemplating non-additive modifications of (\ref{indansatz}) and (\ref{indansatzforVee}).
\eexam

\bexam\label{meanfield} (correctly normalized mean field)
Let $\rho_2={N\choose 2}\gamma$ and $\rho=N\mu$, so that $\gamma$ and $\mu$ have integral $1$. Then
$$\gamma\approx\mu\otimes\mu\Leftrightarrow\frac{\rho_2}{{N\choose 2}}\approx\frac{\rho\otimes\rho}{N^2},$$
which is equivalent to
$$\rho_2\approx\frac{1}{2}\left(1-\frac{1}{N}\right)\rho\otimes\rho,$$
where here and below we use the notation $(\rho\otimes\rho')(x,y)=\rho(x)\rho'(y)$, corresponding to the product measure
when interpreting $\rho$, $\rho'$ as measures.

Note that physicists and chemists use
$$\rho_2\approx\frac{1}{2}\rho\otimes\rho,$$
which is justified in the context of macroscopic systems such as an electron gas (where one has taken a limit
$N\to\infty$), but less natural in the context of DFT for atoms and molecules.
\eexam

\bexam (local density approximation)
In a model system, the so-called free electron gas, the pair density can be determined explicitly [Fr97].
In this case the single-particle density is a constant,
\be \label{homogeneity}
    \rho(x)\equiv\overline{\rho},
\ee
and the pair density can be determined to be
\be \label{electrongas}
    \rho_2(x,y) = \frac{1}{2}\overline{\rho}^2\Bigl(1 - \frac{1}{q} h\Bigl((3\overline{\rho}\pi^2)^{1/3}|x-y|\Bigr)^2\Bigr),
\ee
where
$$
    h(s) = \frac{3(\sin s - s\cos s)}{s^3}.
$$
In particular, at long range $|x-y|\to\infty$ statistical independence is correct, but at short range
$|x-y|\to 0$, $\rho_2$ tends to zero in the case of a single spin state, i.e. it vanishes on the diagonal $x=y$,
and to half the size of a statistically independent sample in the (physical) case of two spin states.
Substituting the result (\ref{electrongas}) 
into the formula for $V_{ee}[\rho_2]$ leads 
to the so-called
{\it local density approximation}
\be \label{LDA}
    \tilde{V}_{ee}[\rho] = J[\rho] - c_x \int_{\R^3}\!\!\rho(y)^{4/3}dy,
\ee
where $c_x=\frac34(\frac3{\pi})^{1/3}$. As a heuristic approximation to $V_{ee}$,
this formula goes back to Dirac and Bloch (for a rigorous justification see [Fr97]).
It was widely used in the early days of DFT, following [KS65].
\eexam

\bexam (B3LYP)
Current functionals used in practice, e.g. the `B3LYP' functional of Becke, Lee, Yang and Parr
[Be93, LYP88], rely on -- from a mathematical point of view questionable -- guesses of functional forms
(e.g. local in $\rho$, or local in $\rho$ and $\nabla\rho$), additional terms depending non-locally
on the orbitals in (\ref{KE2}), and careful fitting of parameters to experimental
or high-accuracy-computational data. The resulting expressions are a little too complicated
to write down here. They have led to an accuracy improvement for $E_0^{QM}$ over the local density approximation
of an order of magnitude or so, but not more, with little progress in the last decade despite continuing effort.
\eexam

\section{Optimal transportation for DFT}
We begin with a basic observation. The weight factor in front of $\rho_2$ in $V_{ee}$ in (\ref{vnevee}) is always positive, and largest on the diagonal $x=y$,
so even ``complete anticorrelation'' might be a better ansatz than independence (keeping in mind that
the states on which the ansatz needs to be good are the minimizers of a functional which includes $V_{ee}$).

In fact, such a complete anticorrelation is exactly what emerges when one starts from the exact Hohenberg-Kohn functional $F_{HK}$, inserts a semiclassical factor $\hbar^2$
in front of the kinetic energy functional $T$ in (\ref{FHK}) and passes to the semiclassical limit $\hbar\to 0$.
In this limit, the Hohenberg-Kohn  functional $F_{HK}$ reduces to the following functional obtained by
a minimization over pair densities instead of wave functions,
\beq \label{formallimit1}
   \tilde{F}[\rho] = \inf_{\rho_2\in\calR_2,\, \rho_2 \mapsto \rho}
   \int_{\R^6} \frac{1}{|x-y|}\, \rho_2(x,y) \, dx\, dy.
\eeq
Here $\rho_2\mapsto\rho$ means that $\rho_2$ satisfies eq. (\ref{marginal}), i.e. it has right and left
marginal $\rho$, and the set $\calR_2$ of admissbile pair densities is the image of $\cala$ under the map $\Psi\mapsto
\rho_2$. Unlike the corresponding admissible set of single-particle densities, $\calR_2$ is not known
explicitly (this is a variant of the representability problem for two-particle density matrices [CY02]).

{\it Formally}, ignoring this point and discarding in particular the smoothness restriction (which can be proved analogously
to the proof in [Li83] of (\ref{admdens2})) that
\beq \label{smoothnessR2}
    \rho_2\in\calR_2 \Longrightarrow \sqrt{\rho_2}\in H^1(\R^6),
\eeq
the above functional $\tilde{F}$ reduces to the functional
\beq \label{EMT}
    E_{OT}[\rho] = \inf_{\rho_2\in\calM_+, \, \rho_2\mapsto\rho} C[\rho_2], \;\;\;\;\;
    C[\rho_2] := \int_{\R^6}\frac{1}{|x-y|} d\rho_2(x,y),
\eeq
where $\calM_+$ denotes the set of (nonnegative) Radon measures on $\R^6$. For a rigorous justification
that (\ref{EMT}) is indeed the correct semiclassical limit of $F_{HK}$ in case $N=2$ see section~5.

The variational problem that has appeared here,
to minimize a ``cost functional'' $C$ over a set of joint measures on $\R^6$ subject to fixed marginals,
is an {\it optimal transport problem}. This type of problem, with ``cost functions'' such as $|x-y|$ or $|x-y|^2$
instead of $|x-y|^{-1}$, dates back to Monge in 1781 and has a famous history, which is nicely summarized in the very readable paper [GM96],
which was our main source when studying the problem (\ref{EMT}).

Before formulating our particular problem we give some notations and definitions.

For a set $Z\subset\R^d$, we denote by $\calp(Z)$ the set of probability measures on $Z$. If $Z$ is a closed subset of $\R^d$ and $\gamma\in\calp(Z)$, then the support of $\gamma$ is the smallest closed set $supp~\gamma\subset Z$ of full mass, that is such that $\gamma(supp~\gamma)=\gamma(Z)=1$.

Suppose $X,Y\subset\R^d$ are closed sets. If $\mu\in\calp(X)$ and $\nu\in\calp(Y)$, we denote by $\Gamma(\mu,\nu)$ the joint probability measures $\gamma$ on $\R^d\times\R^d$ which have $\mu$ and $\nu$ as their marginals, that is with $\mu(U)=\gamma(U\times\R^d)$ and $\nu(U)=\gamma(\R^d\times U)$ for Borel $U\subset\R^d$. In fact, if $\gamma\in\Gamma(\mu,\nu)$, then $supp~\gamma\subset X\times Y$. Typically $\mu$ has density $u_1$ and $\nu$ has density
$u_2$, where $u_1$, $u_2$ are $L^1$ functions, in which case we write
$\Gamma(u_1,u_2)$.

\begin{remark}
In this section it is convenient to eliminate the prefactors in $\rho_2$ and $\rho$, and to consider the cost functional $C$ on probability measures $\gamma$, with equal marginals $\mu$
(which are again probability measures),
with $\rho_2$ in (\ref{EMT}) corresponding to ${N\choose 2}\gamma$ and $\rho$ corresponding to $N\mu$ as in Example~\ref{meanfield}.
\end{remark}

We first formulate the general problem, which is now called \textit{the Kantorovich problem}.
For some \textit{cost function} $c:\R^d\times\R^d\rightarrow\R\cup\{+\infty\}$
we are interested in minimizing the transport cost
\begin{equation} \label{Cofgamma}
C[\gamma]:=\int c(x,y)d\gamma(x,y),
\end{equation}
among joint measures $\gamma\in\Gamma(\mu,\nu)$, called \textit{transport plans}, to obtain
\begin{equation}
\inf_{\gamma\in\Gamma(\mu,\nu)}C[\gamma].
\end{equation}
Let $\calt(\mu,\nu)$ be the set of Borel maps $T:\Omega\subset\R^d\rightarrow\R^d$
that {\em push $\mu$ forward to $\nu$}, i.e. $T_{\#}\mu[V]:=\mu(T^{-1}(V))=\nu(V)$ for Borel $V\subset\R^d$.
The so-called \textit{Monge problem} is to minimize
\beam\label{pushfor}
I[T]=\int_{\bbr^d} c(x,T(x)) d\mu(x)
\eeam
over maps $T$ in $\calt(\mu,\nu)$, called \textit{transport maps}, to obtain
\beam\label{monge}
\inf_{T_{\#}\mu=\nu} I[T].
\eeam
There is a natural embedding which associates to each transport map $T\in\calt(\mu,\nu)$ a transport map $\ga_T:=({\rm id}\times T)_{\#}\mu\in\Gamma(\mu,\nu)$ or, in physics notation, $\gamma_T(x,y):=\delta_{T(x)}(y)\mu(x)$, where id$:\bbr\to\bbr$ is the identity map.
Since $C[\gamma_T]=I[T]$, we conclude that
\beam\label{CI}
\inf_{\ga\in\Ga(\mu,\nu)} C[\ga]\le\inf_{T\in\calt(\mu,\nu)} I[T].
\eeam
Let $\bar\R:=\R\cup\{\pm\infty\}$ and endow $\bar\R$ with the usual topology so that $c\in C(\R^d\times\R^d,\bar\R)$ means that \mbox{$\lim_{(x,y)\to(\bar x,\bar y)} c(x,y)= c(\bar x,\bar y)$}. In particular, if $c(\bar x,\bar y)=+\infty$, then $c(x,y)$ tends to $+\infty$ as $(x,y)$ tends to $(\bar x,\bar y)$, and similarly for limit $-\infty$.

Throughout we will work with cost functions $c$ such that $c(x,y)=h(x-y)$ for all $x,y\in\bbr^d$. We recall the definition of the dual of a convex function and refer to Rockafellar~\cite{Ro72} for standard definitions and further background.
\begin{defn}
The dual (or Legendre transform) $h^*:\R^d\rightarrow\R\cup\{+\infty\}$ of a convex function $h:\R^d\rightarrow\R\cup\{+\infty\}$ is given by
\begin{equation}
\label{legendre}
h^*(y):=\sup_{x\in\R^d}\{<x,y>-h(x)\}.
\end{equation}
\end{defn}
Since our cost functions of interest are not convex, we need to work with the following more subtle definition of Legendre transform.
\begin{defn} (generalized Legendre transform)
Suppose that $l:\bbr\to\bbr\cup\{+\infty\}$ is lower semi-continuous and convex.
Define $k:\R\rightarrow\R\cup \{+\infty\}$ by $k(\lambda)=l(\lambda)$ if $\lambda\ge 0$ and $k(\lambda)=+\infty$ otherwise.
Define $h^*:\R^d\rightarrow\R\cup \{+\infty\}$ by
\begin{equation}
\label{legr}
h^*(x)=k^*(-|x|):=\sup_{\beta \in\bbr}\{-\beta |x|-k(|x|)\},\quad x\in\R^d.
\end{equation}
We define $l^{\circ}(\lambda)=k^*(-|\lambda|)$ for $\la\in\R$.
\end{defn}
\begin{defn}
A function $\psi:\R^d\rightarrow\R\cup\{-\infty\}$, not identically $-\infty$, is said to be {\em $c$-concave} if it is the infimum of a family of translates and shifts of $h(x)$: i.e, there is a set $A\subset\R^d\times\R$ such that
\begin{equation}
\label{cconcave}
\psi(x):=\inf_{(y,\lambda)\in A}\{ c(x,y)+\lambda\},\quad x\in \bbr^d.
\end{equation}
\end{defn}

Let $\Delta:=\{(x,x)\mid x\in\R^d\}$ and $c:\R^d\times \R^d\rightarrow\R\cup\{+\infty\}$ be such that $c(x,y):=h(x-y):=l(|x-y|)\ge 0$ and with $c$ and $l$ such that
\begin{enumerate}
\item [(A1)]
$l:[0,\infty]\rightarrow [0,\infty]$ is strictly convex, strictly decreasing and $C^1$ on~$(0,\infty)$;
\item [(A2)] $c\in C^1(\R^d\times\R^d\setminus\Delta,\R)$;
\item [(A3)] $c:\R^d\times\R^d\rightarrow [0,+\infty]$ is lower semi-continuous;
\item [(A4)] for every $x_0\in\R^d, c(x_0,x_0)=+\infty$.
\end{enumerate}

\begin{remark}
Note that \cite{GM96} only assume that $l$ is continuous and not $C^1$; using their arguments, we could replace (A1) and (A2) by
\begin{enumerate}
\item [{\rm (A1')}]
$l:[0,\infty]\rightarrow [0,\infty]$ is strictly convex, decreasing and continuous on~$(0,\infty)$;
\item [{\rm (A2')}] $c\in C(\R^d\times\R^d\setminus\Delta,\R)$.
\end{enumerate}
\end{remark}
The main result of this section, obtained by combining the results from Theorems \ref{unique} and \ref{geom}, is
\begin{theorem}
\label{uniqgeom}
Assume that $c$ and $l$ satisfy (A1)-(A4), and let $\mu,\nu\in \calp(\R^d)$ be absolutely continuous with respect to the Lebesgue measure.
Then there exists a unique minimizer $\gamma\in\Gamma(\mu,\nu)$ of the functional $C$ defined in (\ref{Cofgamma}), and
a unique transport map $T$ pushing $\mu$ forward to $\nu$ such that
$\gamma=(id,T)_{\#}\mu$ (or, in physics notation, $\gamma(x,y)=\delta_{T(x)}(y)\mu(y)$). This map
is of the form $T(x)=x-\nabla h^{*}(\nabla\psi(x))$ for some $c$-concave function $\psi$ on $\R^d$.
\end{theorem}
We next give a simple example, which illustrates the emergence of such an optimal map $T$ when the marginals consist of two Dirac delta
functions, and explicitly compute the function $h^*$ appearing above for our cost function of interest.

\bexam
Let $a,b\in \R$. For $c$ and $l$ as in (A1)-(A4) we are interested in minimizing
\begin{equation}
\label{minim}
\int c(x,y)d\gamma(x,y),
\end{equation}
subject to
\begin{equation}
\label{constr}
\int\gamma(x,y)dx=\delta_a(y)+\delta_b(y~)~\mbox{and}~~\int\gamma(x,y)dy=\delta_a(x)+\delta_b(x),
\end{equation}
where $\delta_a$ is the Dirac function such that, for all Borel subsets of $\R$, we have $\int_{\Omega}\delta_a(y)=1$ if $a\in\Omega$,
and $\int_{\Omega}\delta_a(y)=0$ otherwise. We claim that the minimum in (\ref{constr}) is attained for
\begin{equation}
\label{formdirac}
\gamma(x,y)=\delta_a(x)\delta_b(y)+\delta_b(x)\delta_a(y) =: \gamma_0(x,y).
\end{equation}
To show this, note first that
$$\gamma(x,y)=c_{aa}\delta_a(x)\delta_a(y)+c_{ab}\delta_a(x)\delta_b(y)+c_{ba}\delta_b(x)\delta_a(y)+c_{bb}\delta_b(x)\delta_b(y),$$
with $c_{aa},c_{ab},c_{ba},c_{bb}\ge 0$ and $c_{aa}+c_{ab}+c_{ba}+c_{bb}=2$. Due to the constraints on $\gamma$ from (\ref{constr}) we have
$c_{ba}=c_{ab}$ and $c_{aa}=c_{bb}$. Hence
$$\gamma(x,y)=\alpha\left(\delta_a(x)\delta_a(y)+\delta_b(x)\delta_b(y)\right)+\beta\left(\delta_a(x)\delta_b(y)+\delta_b(x)\delta_a(y)\right),~\mbox{with}~\alpha,\beta\ge 0~\mbox{and}~\alpha+\beta=1. $$
Minimizing (\ref{minim}) subject to (\ref{constr}) is then equivalent to the following problem:
$$ \mbox{Minimize }2\alpha\, l(0)+2\beta\, l(|b-a|) \mbox{ subject to the constraints }\alpha,\beta\ge 0, \, \alpha+\beta=1.$$
Since by (A1) we have $l(0)>l(|a-b|)$, the minimum in the above is attained for $\alpha=0$ and $\beta=1$, which proves the claim.

Formula (\ref{formdirac}) for the minimizer admits a very interesting interpretation which motivates the notion of {\it optimal transport map}
and foresees the structure of general minimizers as given in Theorem \ref{tdescrip}. Denote the single-particle density $\delta_a(x)+\delta_b(x)$
in (\ref{constr}) by $\rho(x)$, and introduce the map $T\, : \, \{a,b\}\to\{a,b\}$ which maps $a$ to $b$ and $b$ to $a$. Then $T$ pushes $\rho$
forward to $\rho$, and the optimal measure $\gamma_0$ has the form
$$
   \gamma_0(x,y) = \delta_{T(x)}(y)\rho(x),
$$
or, in measure-theoretic notation,
$$
   \gamma_0 = (id,T)_{\#}\rho.
$$
\eexam

\bexam (generalized Legendre transform for Coulomb cost) Let $h(x)=k(|x-y|)$ with
$k(\lambda)=\lambda^{-1}$ for $\lambda>0$ and $+\infty$ for $\lambda\le
0$. Note that $k \, : \, \R\to\R\cup\{+\infty\}$
is lower semi-continuous and convex. The ordinary Legendre transform of
$k$ is given by
$k^*(s)=\sup_{\beta\in\R}\{\beta s - k(\beta)\}$. Since $k(\beta)=+\infty$
for $\beta\le 0$, negative values of $\beta$ do not contribute to the
above supremum (the term in brackets then being $-\infty$). Consequently,
$$
   k^*(s) = \sup_{\beta > 0} \{\beta s - k(\beta)\}.
$$
When $s>0$, we infer $k^*(s)=+\infty$. It thus remains to calculate
$$
   k^*(-s) = \sup_{\beta > 0}\{-\beta s - k(\beta)\},\quad s\ge 0.
$$
Recall now that $k(\beta)=1/\beta$. The elementary calculus problem of
maximising
the function in brackets on $(0,\infty)$ has the unique solution
$\beta=1/\sqrt{s}$, whence
$k^*(-s)=-2\sqrt{s}$ for $s\ge 0$. Altogether it follows that the
generalized Legendre transform
of $h$ is
\be \label{legtrf}
    h^*(x) = - 2 \sqrt{|x|},\quad x\in\R^d.
\ee
Consequently, if $c(x,y)=\frac1{|x-y|}$ then the optimal map $T$ in Theorem \ref{uniqgeom} is of the form
$T(x) = x + \frac{\nabla\psi(x)}{|\nabla\psi(x)|^{3/2}}$.
\eexam

Subsections 3.1-3.4 are devoted to the proof of Theorem \ref{uniqgeom}. The proofs follow partially the proofs in
\cite{GM96}, \cite{KM07}, \cite{GM95} and \cite{GO07}. In the final subsection, 3.5, we derive some general properties of the unique optimal measure
and the optimal cost under the assumption of equal marginals, $\mu=\nu$.

\subsection{Definitions and notation}

In the following we present some definitions which are needed throughout the remainder of the section.

\begin{defn}\label{ctransf}
\begin{enumerate}
\item[$(1)$]
Let $V\subset\R^d$. A $c$-concave function $\psi: \R^d\rightarrow \R\cup\{-\infty\}$ is said to be the {\em $c$-transform} on $V$ of a function $\phi$ if (\ref{cconcave}) holds with $A\subset V\times\R$. Moreover,
$$\psi(x)=\inf_{y\in V}\{c(x,y)-\phi(y)\},$$
for some function $\phi:V\rightarrow\R\cup\{-\infty\}$.
\item[$(2)$]
The $c$-transform of a function $\psi :\R^d\rightarrow\R\cup\{-\infty\}$ is the function $\psi^c:\R^d\rightarrow\R\cup\{-\infty\}$ defined by
$$\psi^c(y)=\inf_{x\in\R^d}\{c(x,y)-\psi(x)\}.$$
\item[$(3)$]
A subset $S\subset X\times Y$ is called {\em $c$-cyclically monotone}, if for any finite number of points $(x_j,y_j)\in S, j=1,\ldots,n$ and permutations
$\sigma \, : \, \{1,..,n\}\to\{1,..,n\}$
\begin{equation}
\label{cycleqn}
\sum_{j=1}^n c(x_j,y_j)\le\sum_{j=1}^n c(x_{\sigma(j)},y_j).
\end{equation}
\end{enumerate}
\end{defn}

\begin{remark}
Remark~3.4 of \cite{GO07} proves the following: If $\psi:\bbr^d\to\R\cup\{-\infty\}$ is not identically $-\infty$, and is given by \eqref{cconcave}, then we have\\
(i) $\psi(y)\ge -\la>-\infty$ for all $(y,\la)\in A$, where $A$ is the set in \eqref{cconcave}. Hence, $\psi^c$ is not identically $-\infty$.\\
(ii) $\psi^{cc}=\psi$.
\end{remark}

\begin{theorem}[optimal measures have c-cyclically monotone support: Proposition 3.2 of \cite{GO07}]\label{3.2}
Assume that $X,Y\in\R^d$ are closed sets, that $\mu\in\calp(X)$,
$\nu\in\calp(Y)$ and that $c\ge 0$ is lower semicontinuous on $X\times Y$.
Then the following hold:
\begin{enumerate}
\item [{\rm (a)}] There is at least one optimal measure $\gamma\in\Gamma(\mu,\nu)$.
\item [{\rm (b)}] Suppose that in addition $c\in C(X\times Y,\bar\R)$. Unless $C\equiv +\infty$ throughout $\Gamma(\mu,\nu)$, there is a $c$-cyclically monotone set $S\subset \R^d\times\R^d$ containing the support
of all optimal measures in $\Gamma(\mu,\nu)$.
\end{enumerate}
\end{theorem}

\begin{defn}
\begin{enumerate}
\item[$(1)$]
A function $\psi:\R^d\rightarrow\R\cup\{-\infty\}$ is {\em superdifferentiable at $x\in\R^d$}, if $\psi(x)$ is finite and there exists $y\in\R^d$ such that
\begin{equation}
\label{supp}
\psi(x+z)\le\psi(x)+<z,y>+o(|z|) \mbox{ as }|z|\to 0;
\end{equation}
here $o(\lambda)$ means terms $\eta(\lambda)$ such that $\eta(\lambda)/\lambda$ tends to zero with $\lambda$.
\item[$(2)$]
A pair $(x,y)$ {\em belongs to the superdifferential $\partial^\cdot\psi\subset \R^d\times \R^d$ of $\psi$}, if $\psi(x)$ is finite and (\ref{supp}) holds, in which case $y$ is called a {\em supergradient of $\psi$ at $x$}.
Such supergradients $y$ comprise the set $\partial^\cdot\psi(x)\subset\R^d$, while for $V\subset\R^d$ we define $\partial^\cdot\psi(V):=\cup_{x\in V}\partial^\cdot\psi(x)$.
\item[$(3)$]
The analogous notions of {\em subdifferentiability}, {\em subgradients} and the {\em subdifferential} $\partial_\cdot\psi$ are defined by reversing inequality (\ref{supp}).
\item[$(4)$]
A real-valued function $\psi$ will be differentiabale at $x$ precisely if it is both super- and subdifferentiable there; then
$$\partial^\cdot\psi(x)=\partial_\cdot\psi(x)=\{\nabla\psi(x)\}.$$
\end{enumerate}
\end{defn}

\begin{defn}
The {\em $c$-superdifferential} $\partial^c\psi$ of $\psi:\R^d\rightarrow \R\cup\{-\infty\}$, not identical $-\infty$, consists of the pairs $(x,y)\in \R^d\times\R^d$ for which $c(x,y)-\psi(x)\leq c(z,y)-\psi(z)$ for all $z\in \R^d$.
\end{defn}

\begin{lemma}[relating $c$-superdifferentials to subdifferentials: Lemma 3.1 of \cite{GM96}]
\label{difsubdif}
Let $h:\R^d\rightarrow\R^d$ and $\psi:\R^d\rightarrow\R\cup\{-\infty\}$. If $c(x,y)=h(x-y)$, then $(x,y)\in\partial^c\psi$ implies $\partial.\psi(x)\subset\partial.h(x-y)$.
When $h$ and $\psi$ are differentiable, then $\nabla\psi(x)=\nabla h(x-y)$.
\end{lemma}

\begin{defn}
A function $\psi:\R^d\rightarrow\R\cup\{-\infty\}$ is said to be {\em locally semi-concave} {\em (locally semi-convex)} at $p\in\R^d$, if there is a constant $\lambda<\infty$, which makes $\psi(x)-\lambda|x|^2$ concave {\em (convex)} on some (small) open ball centered at $p$.
\end{defn}

\begin{remark}\noindent
\label{diagzero}
Suppose that $\mu\in\calp(X)$ and $\nu\in\calp(Y)$ have no atoms and that $\gamma^*$ minimizes $C[\gamma]$ over $\Gamma(\mu,\nu)$ and that $C(\gamma^*)<\infty$. Then $\gamma^*(\Delta)=0$ and so, $supp~\gamma^*\setminus\Delta$ contains at least one element, say $(x_0,y_0)$. Also $\gamma^*(E)=0$, where $E=(x_0\times Y)\cup (X\times y_0)$. Hence the set $X\times Y\setminus (E\cap\Delta)$ is non-empty, so it contains an element $(\bar{x}_0,\bar{y}_0)$. Note that $x_0,y_0\notin \{\bar{x}_0,\bar{y}_0\}$.
\end{remark}
We will need the non-atomic property of the marginal measures $\mu$ and $\nu$ (and the resulting remark above) from the uniqueness section~\ref{uniq} onwards, as a means to bypass the singularity of $c$ on the diagonal. We will use Remark \ref{diagzero} in Lemma \ref{newsupdif} and in Lemma \ref{semiconc} below.

\begin{remark}
\label{diagmonzero}
Suppose that $S\subset X\times Y$ is $c$-cyclically monotone and contains two pairs $(x_0,y_0), (\bar{x}_0,\bar{y}_0)$ such that $x_0\neq y_0$, $\bar{x}_0\neq\bar{y}_0$ and $\bar{x}_0,\bar{y}_0\notin \{x_0, y_0\}$. Then for all $(x,y)\in S$, we have $x\neq y$.
\end{remark}
\begin{proof}
Let us assume that $(x,y)\in S$, with $x=y$.
There are the following possibilities to consider:  $x=y=x_0$,  $x=y=y_0$, $x=y={\bar x}_0$ $x=y={\bar y}_0$ and $x,y\notin\{x_0, y_0,\bar x_0, \bar y_0\}$.

We present a proof for one of the cases, the other cases being treated analogously.
Consider $x=y=y_0$. Then $x,y\notin\{x_0,\bar{x}_0,\bar{y}_0\}$ and from (\ref{cycleqn}), we get
$$c(x,y)\le c(x,{\bar y}_0)+c(x_0,y)+c({\bar x}_0,y_0)-c(x_0,y_0)-c({\bar x}_0,{\bar y}_0)<+\infty,$$
which leads to a contradiction as $c(x,y)=+\infty$.
\end{proof}

\subsection{Existence of an optimal measure with $c$-cyclical monotone support}

The main issue in this section is not the existence of an optimal measure, as the existence of an optimal measure is assured by Theorem
\ref{3.2} (a), but the existence of an optimal measure with $c$-cyclical monotone support. In order to use Theorem \ref{3.2} (b) and
construct such an optimal measure, we need to first construct a joint measure $\gamma$, with marginals $\mu$ and $\nu$ and with
$C[\gamma]<\infty$. This is done in the Lemma below.

\begin{lemma}
\label{exist1}
Suppose that $\mu,\nu\in\calp(\R^d)$ and are absolutely continuous with respect to the Lebesgue measure. Then there exists $\gamma\in\Gamma(\mu,\nu)$ and $\epsilon>0$ such that for all $(x,y)\in supp~\gamma$,
$$|x-y|>\epsilon$$
\end{lemma}

\begin{proof}
The proof follows similar arguments as the proof of Proposition 4.1 from \cite{GO07} adapted to our situation.

Let $b,c\in\R_+$ and let $S(0;b):=\{x\in\R^d:|x|\le b\}$, $S^c(0;b):=\{x\in\R^d:|x|>b\}$ and $S(b;c):=\{x\in\R^d:b<|x|\le c\}$. Since $\mu$ and $\nu$ are absolutely continuous with respect to the Lebesgue measure, the functions
\begin{equation}
\label{cont}
t\mapsto\mu|_{S(0;t)},~t\mapsto\nu|_{S(0;t)},~t\mapsto\mu|_{S^c(0;t)},~t\mapsto\nu|_{S^c(0;t)}
\end{equation}
are continuous.
\\
\textbf{Step 1.}
We assume first that there exists $b\in\R_+$ such that $supp~\mu\subset S(0;b)$ and $supp~\nu\subset S^c(0;b)$. Then, because of (\ref{cont}), we may choose $\epsilon_1,\epsilon_2>0$ such that
$$\mu(S(0;b-\epsilon_1))=\nu(S(b;b+\epsilon_2))=\frac{1}{2}~\mbox{with}~\epsilon_1<b.$$
Let
$$\mu^{-}=\mu|_{S(0;b-\epsilon_1)},~\mu^{+}=\mu|_{S(b-\epsilon_1;b)},~\nu^{-}=\nu|_{S(b:b+\epsilon_2)}~\mbox{and}~\nu^{+}=\nu|_{S^c(0;b+\epsilon_2)}.$$
Set
$$\gamma:=2(\mu^{-}\otimes\nu^{-}+\mu^{+}\otimes\nu^{+}).$$
Note that $\gamma\in\Gamma(\mu,\nu)$ and for all $(x,y)\in supp~\gamma$, we have
$$|x-y|\ge\min\{\epsilon_1,\epsilon_2\}>0.$$
\\
\textbf{Step 2.} Assume that $\mu$ and $\nu$ are arbitrary in $\calp(\R^d)$.
We use (\ref{cont}) to choose $b\in\R_+$ such that
\begin{equation}
\label{case2}
\mu(S(0;b))=\nu(S^c(0;b))=m.
\end{equation}
If $m=0$, then we reduce the discussion to Step~1. Similarly if $m=1$.
We can therefore assume that $0<m<1$.
More precisely, if we denote by $f(b):=\mu(S(0;b))-\nu(S^c(0;b))$ for all $b\in\R_+$, then $f(b)$ is an increasing and continuous function of $b$, going from negative values to positive ones as $b$ goes from $0$ to $+\infty$.
Therefore, there exists $b_0\in\R_+$ such that $f(b_0)=0$, which is equivalent to $\mu(S(0;b_0))=\nu(S^c(0;b_0))$.  Set
$$\mu^{-}=\mu|_{S(0;b_0)},~\mu^{+}=\mu|_{S^c(0;b_0)},~\nu^{-}=\nu|_{S(0;b_0)}~\mbox{and}~\nu^{+}=\nu|_{S^c(0;b_0)}.$$
By (\ref{case2}), $\frac{\mu^{-}}{m}$ and $\frac{\nu^{+}}{m}$ are probability measures. They satisfy
$$supp~\mu^{-}\subset S(0;b_0)~\mbox{and}~supp~\nu^{+}\subset S^c(0;b_0).$$
Therefore, they satisfy the assumptions of Step~1, so there exists $\delta_1>0$ and $\gamma_1\in\Gamma(\frac{\mu^{-}}{m},\frac{\nu^{+}}{m})$ such that for all $(x,y)\in supp~\gamma_1$, we have
$$|x-y|>\delta_1.$$
Similarly, there exists $\delta_2>0$ and $\gamma_2\in\Gamma(\frac{\mu^{+}}{1-m},\frac{\nu^{-}}{1-m})$ such that for all $(x,y)\in supp~\gamma_2$, we have
$$|x-y|>\delta_2.$$
Set $\gamma:=m\gamma_1+(1-m)\gamma_2.$
Then $\gamma\in\Gamma(\mu,\nu)$ and for all $(x,y)\in supp~\Gamma(\mu,\nu)$, we have
$$|x-y|\ge\min\{\delta_1,\delta_2\}.$$
\end{proof}

\begin{theorem}[existence of optimal measure with $c$-cyclical monotone support]
\label{exist}
Assume that $c(x,y)=h(x-y):=l(|x-y|)$ satisfies (A1)--(A4) and
let $\mu,\nu\in\calp(\R^d)$ be absolutely continuous with respect to the Lebesgue measure. Then there exists a measure $\gamma_0$
with $c$-cyclically monotone support which minimizes the functional $C(\gamma)$ introduced in (\ref{Cofgamma}) over $\Gamma(\mu,\nu)$.
\end{theorem}
\begin{proof}
By Lemma~\ref{exist1}, there exists $\gamma\in\Gamma(\mu,\nu)$ and $\epsilon>0$ such that
\begin{equation}
\label{37}
|x-y|>\epsilon
\end{equation}
for all $x,y\in supp~\gamma$. Since $l$ is strictly decreasing on $(0,\infty)$, (\ref{37}) together with (A1) ensure that $c$ is uniformly bounded from above on $supp~\gamma$ by $l(\epsilon)$. This proves
that $C[\gamma]<\infty$. The statement follows now immediately from Theorem~\ref{3.2}.
\end{proof}

\subsection{Geometrical characterization of the optimal measure}
\label{uniq}

It is well known that a set is cyclically monotone if and only if it is contained in the subdifferential of a $c$-concave function; this result was proved for general cost functions $c : X\times Y\rightarrow\R$ in \cite{SK92}.
The following theorem is a further extension that is needed to deal with cost functions which satisfy (A1)-(A3) (and are allowed to take the value $+\infty$ somewhere).

\begin{lemma}\label{newsupdif} Suppose that $X,Y\subset \R^d$ are closed sets.
\begin{itemize}
\item [(1)] 
For $S\subset X\times Y$ to be c-cyclically monotone, it is necessary and sufficient that $S\subset\partial^c\psi$ for some $c$-concave $\psi:X\rightarrow \R\cup\{-\infty\}$.
\item [(2)] Suppose that $S\subset X\times Y$ is $c$-cyclically monotone and contains two pairs $(x_0,y_0), (\bar{x}_0,\bar{y}_0)$ such that $\bar{x}_0,\bar{y}_0\notin \{x_0, y_0\}$. Let $\psi :\R^d \rightarrow \R\cup\{-\infty\}$ be the $c$-concave function from (1).
Then
\begin{itemize}
\item [(2a)] $\psi(x_0)$ and $\psi(\bar{x}_0)$ are finite,
\item [(2b)] whenever $(x,y)\in\partial^c\psi$, we have that $\psi(x)>-\infty$ and $x\neq y$,
\item [(2c)] we have $\mu$-a.e.
$$\psi(x)=\inf_{y\in Y}\{c(x,y)-\psi^c(y),x\neq y\},\quad x\in X.$$
\end{itemize}
\end{itemize}
\end{lemma}
\begin{proof}
Part (1) has been proved in Theorem~2.7 of \cite{GM96} or Lemma~2.1 of \cite{Ru96}.\\
(2) As in the proof of Lemma~2.1 in \cite{Ru96}, we define
\begin{equation}
\label{ctransfdef}
\psi(x)=\inf_{n\in\N}\inf_{\{(x_j,y_j)\}_{j=0}^n\subset S}\left\{c(x,y_n)+\sum_{j=0}^{n-1}c(x_{j+1},y_j)-\sum_{j=0}^nc(x_j,y_j)\right\},\quad x\in X.
\end{equation}
(2a) The construction above yields a $c$-concave $\psi$ on $X$ with $S\subset\partial^c\psi$. Standard arguments (see for example \cite{Ru96}) give that $\partial^c\psi$ is $c$-cyclically monotone in the sense that
\beam
\sum_{j=0}^n c(x_j,y_j)\le \sum_{j=0}^{n-1} c(x_{n+1},y_n)+ c(x_{n+1},y_n)\quad\mbox{with }x_{n+1}=x_0.
\eeam
and so $\psi(x_0)\ge 0$.
Taking $n=1$, $x_1=\bar x_0$ and $y_1=\bar y_1$ in (\ref{ctransfdef}) gives that $\psi(x_0)\le 0$. We conclude that $\psi(x_0)=0$ and so $\psi$ is finite and not identically $-\infty$. Note that by construction, $\psi$ is the $c$-concave function from (1).
Recall now that $(x,y)\in\partial^c\psi$ is equivalent to
\begin{equation}
\label{cconcequiv}
c(x,y)-\psi(x) \le c(z,y)-\psi(z),\quad z\in X.
\end{equation}
Now recall Remark~\ref{diagzero}.
Setting $(x,y)=({\bar x}_0, {\bar y}_0)$, $z=x_0$ in (\ref{cconcequiv}), using the fact that $\psi(x_0)=0$ and that $x_0\neq{\bar y}_0$, we obtain that $\psi({\bar x}_0)$ is finite.\\
(2b)
Next, if $(u_0,y)\in S$, setting $z=u_0$, we have that
\begin{equation}\label{29}
c(x,y)-\psi(x)\le c(u_0,y)-\psi(u_0).
\end{equation}
If $y\neq x_0$, we set $u_0=x_0$ in (\ref{29}) to obtain the claim.
If $y=x_0$, we set $u_0=\bar{x}_0$ and we use the fact that $\psi(\bar{x}_0)$ is finite, to obtain the claim.
\\
(2c) This representation is a simple consequence of the construction.
\end{proof}

We have proved the following. If we define
\begin{equation}
\partial^c_0\psi:=\{(x,y)\in\partial^c\psi\mid x\neq y\}
\quad\mbox{and}\quad
\partial^c_0\psi(x):=\partial^c\psi(x)\setminus\{x\},
\end{equation}
then $\partial_0^c\psi=\partial^c\psi$ $\mu$-a.e.
and we will focus on the off-diagonal elements from now on. We also assume for the rest of this chapter that $X=Y=\R^d$.

\begin{lemma}[$\mu$-a.e. differentiability of $c$-transforms]\label{semiconc}
Let $c$ and $l$ satisfy (A1)-(A4).
Then the function $\psi$ from Lemma \ref{newsupdif} is $\mu$-a.e. differentiable on $\R^d$.
\end{lemma}
\begin{proof}
Recall that
$$\psi(x)=\inf_{y\in\R^d} \{c(x,y)-\psi^c(y),x\neq y\}.$$
We will prove that $\psi$ is $\mu$-a.e. differentiable on $\R^d$.
\\
\textbf{Step 0.}
Let $r>0$ be such that $\psi$ takes finite values at two or more points in $S(0;r)$ (this is possible due to Remark \ref{diagzero} and Lemma~\ref{newsupdif}). Let $0<a<r$ arbitrarily fixed, which means that $S(0;a)\subset S(0;r)$. We will show in Steps $1-3$ below that $\psi$ is $\mu$-a.e. differentiable on $S(0;a)$, from which we will derive in Step 4 the corresponding property on $\R^d$. The reason for the choice of an $0<a<r$ such that $S(0;a)\subset S(0;r)$, will become apparent in Step $2$ below.

In order to prove that $\psi$ is $\mu$-a.e. differentiable on $S(0;a)$, take $x\in S(0;a)$ arbitrarily fixed.
Then
$$\psi(x)=\min\{\psi_1(x),\psi_2(x)\},$$
where
$$\psi_1(x)=\inf_{y\in S(0;r)} \{c(x,y)-\psi^c(y),x\neq y\}
\quad\mbox{and}\quad
\psi_2(x)=\inf_{y\in S^c(0;r)}\{c(x,y)-\psi^c(y),x\neq y\}.$$
Due to the fact that $\psi$ takes finite values at two or more points in $S(0;r)$, it follows from the definition that $\psi_1$ and $\psi_2$ also take finite values at two or more points in $S(0;r)$.

\textbf{Step 1. $\boldsymbol{\psi_1}$ is locally Lipschitz and semi-concave on $\boldsymbol{S(0;a)}$:}

As $\psi_1$ takes finite values at two or more points, the proof follows the same reasoning as the proof of Proposition A.6 in \cite{KM07}
and will be omitted. Note also that by Proposition~C.6 of \cite{GM96}, differentiability of $\psi_1$ can only fail on a set of $\mu$-measure zero.

\textbf{Step 2. $\boldsymbol{\psi_2}$ is locally Lipschitz and $\boldsymbol{\mu}$-a.e. differentiable on $\boldsymbol{S(0;a)}$:}

Let $\delta:=r-a>0$. Define $\xi\le 0$ by using
the right derivative $2\delta\xi:=l'(\delta^+)$ of $l$ at $\delta$. Then the function $l_{\delta}(\lambda)=l(\lambda)-\xi\lambda^2$ is strictly convex on $[\delta,\infty)$ and non-decreasing since $l'_\delta(\delta^+)=0$. Extend this function to $\lambda\le\delta$ by making $l_\delta(\lambda)$ constant-valued there. Then $h_\delta(x):=l_\delta(|x|)$ will be convex on $\R^d$: taking $x,y\in \R^d$ and $0<t<1$ implies
\begin{equation}
h_\delta((1-t)x+ty)\le l_\delta((1-t)|x|+t|y|)\le (1-t)h_\delta(x)+th_\delta(y).
\end{equation}
Note that $h(x)=h_\delta(x)+\xi|x|^2$ whenever $|x|\ge\delta$. Take $x\in S(0;a)$ and $y\in S^c(0;r)$; then $|x-y|\ge\delta$ and the definition of $\psi_2$ yields
$$\psi_2(x)-\xi|x|^2=\inf_{y\in S^c(0;r)}\Big\{ h_\delta (x-y)+2\xi<x,y>+\xi y^2-\psi^c(y)\Big\}.$$
Note now that $h_\delta$ satisfies conditions (H1)-(H3) of \cite{GM96}. Since $l$ is continuous on $(0,\infty)$, we can apply Theorem~3.3 of \cite{GM96} to $h_\delta$ and $\psi_2$ (see also Proposition C.2 in \cite{GM96}) and thus, $\psi_2(x)-\xi|x|^2$ will be locally Lipschitz. Using the fact that $\psi_2(x)-\xi|x|^2$ is locally Lipschitz, Rademacher's theorem shows that the gradient $\nabla\psi_2$ is defined $\mu$-a.e. everywhere on $S(0;r)$.

\textbf{Step 3. $\boldsymbol{\psi}$ is $\boldsymbol{\mu}$-a.e. differentiable on $\boldsymbol{S(0;a)}$:}

Since $\psi(x)=\min\left(\psi_1(x),\psi_2(x)\right)$, the $\mu$-a.e. differentiability of $\psi$ on $S(0;a)$ follows immediately as $\psi$ is the minimum of two $\mu$-a.e. differentiable functions.

\textbf{Step 4. $\boldsymbol{\psi}$ is $\boldsymbol{\mu}$-a.e. differentiable on $\boldsymbol{\R^d}$:}

Let $(a_n)_{n\in\N}$ be an increasing sequence  of positive real numbers tending to infinity as $n\to\infty$ and $\R^d=\cup_{n\in\N} S(0;a_n)$. Let $A:=\{x\in\R^d:\psi~\mbox{is differentiable at}~x\}$.  Then $\mu(A)=\lim_{n\rightarrow\infty}\mu(A\cap S(0;a_n))$. The statement follows now immediately by means of Step~3.
\end{proof}

The following is a version of Lemma~5.2 of \cite{GM96} for strictly convex and decreasing $l$.

\begin{lemma}[the $c$-superdifferential lies in the graph of a map]
\label{supdif}
Let $c$ satisfy (A1)-(A4). Suppose that $\psi:\R^d\rightarrow\R$ is differentiable at some $x\in\R^d$. Then $y\in\partial^c_0\psi(x)$ implies that $h^*$ is differentiable at $\nabla\psi(x)$ and that $y=x-\nabla h^*(\nabla\psi(x))$.
\end{lemma}

\begin{proof}
We recall first that by (A1), $\nabla h$ is injective off the diagonal; the injectivity of $\nabla h$ off the diagonal is \textbf{crucial} for the argument of this lemma, as will become apparent below.
The proof follows similar steps as the proof of Lemma~5.2 in \cite{GM96}. Let $y\in\partial^c_0\psi(x)$.
Then $x\neq y$, so $h$ is differentiable at $x-y$. From Lemma~\ref{difsubdif} (see also Lemma 3.1 of \cite{GM96}) we have that  $\partial_{\cdot}\psi(x)\in\nabla h(x-y)$. As $\psi$ is differentiable at $x$, we have $\partial_{\cdot}\psi(x)=\nabla\psi(x)$ and $\nabla\psi(x)=\nabla h(x-y)$. Since $x\neq y$ and since $l$ is strictly convex and strictly decreasing, the gradient $\nabla\psi(x)$ does not vanish $\mu$-a.e. Lemma~\ref{A6} (ii)-(iii) implies both $(\nabla\psi(x),x-y)\in\partial_\cdot h^*$ and differentiability of $h^*$ at $\nabla\psi(x)$, whence $\nabla h^*(\nabla\psi(x))=x-y$.
\end{proof}

\begin{lemma}
\label{map}
Let $c$ and $l$ satisfy (A1)-(A4) and let $\mu,\nu\in\calp(\R^d)$. Suppose that a joint measure $\gamma\in\Gamma(\mu,\nu)$ has $supp~\gamma\subset \partial^c_0\psi=\partial^c\psi$, where $\psi:\R^d\rightarrow R$ is the $c$-transform of a function on $supp~\nu$. The map $T(x):=x-\nabla h^*(\nabla\psi(x))$ pushes $\mu$ forward to $\nu$. In fact, $\gamma=(id,T)_{\#}\mu$ and $T_{\#}\mu=\nu$.
\end{lemma}
\begin{proof}

The proof follows the same steps of Theorem 5.4 from \cite{GM96}. For the reader's convenience, we provide the reasoning below.

To begin, one would like to know that the map $T(\cdot)$ is Borel and defined $\mu$-a.e.. By Lemma~\ref{semiconc}, differentiability of $\psi$ can only fail on a set ${\mathcal N}$ of $\mu$-measure zero in $\R^d$. Thus
$\mu({\mathcal N})=0$, $\gamma(N\times Y)=0$ and so the map $\nabla\psi$ is defined $\mu$-a.e.. Since by Remark~\ref{diagzero}, $\gamma(\Delta)=0$ and since $supp~\gamma\subset\partial^c\psi=\partial^c_0\psi$, we have $\gamma(\partial^c_0\psi)=1$. Therefore, define $S:=\{(x,y)\in\partial^c_0\psi|x\in dom\nabla\psi\}$, where $dom\nabla\psi$ denotes the subset of $\R^d$ on which $\psi$ is differentiable. Lemma~\ref{semiconc} shows that $\psi$ is $\mu$-a.e. differentiable on $dom~\psi$. Since
its gradient is obtained as the pointwise limit of a sequence of continuous approximants
(finite differences), $\nabla\psi$ is Borel measurable on the (Borel) set $dom~\nabla\psi$ where it
can be defined. Lemma~\ref{A6} shows that $\nabla h^*$ is a Borel map. For $(x,y)\in S$, Lemma~\ref{supdif} implies that $T$ is defined at $x$ and $y=T(x)$. Thus $T$ is defined on the projection of $S$ onto $\R^d$ by $\pi(x,y):=x$; it is a Borel map since $\nabla\psi$ and $\nabla h^{*}$ are. Moreover, the set $\pi(S)$ is Borel and has full measure for $\mu$: both $\partial^c_0\psi$ and $\pi(\partial^c_0\psi)$ are $\sigma$-compact, so $\pi(S)=\pi(\partial^c_0\psi)\cap dom\nabla\psi$ is the intersection of two Borel sets with full measure. Thus $\gamma(Z\cap S)=\gamma(Z)$ for $Z\subset \R^d\times \R^d$. It remains to check that $ (id,T)_{\#}\mu=\gamma$, from which $T_{\#}\mu=\nu$ follows immediately.

It suffices now to show that the measure $(id,T)_{\#}\mu$ coincides with $\gamma$ on products $U\times V$ of Borel sets $U, V\in \R^d$; the semi-algebra of such products generates all Borel sets in $\R^d\times \R^d$. Define $S:=\{(x,y)\in\partial^c_0\psi|x\in dom~\nabla\psi\}$. Therefore, since $y=T(x)$ if $(x,y)\in S$, we have
\begin{equation}
\label{19}
(U\times V)\cap S=((U\cap T^{-1}(V))\times \R^d)\cap S.
\end{equation}
Being the intersection of two sets having full measure for $\gamma$ -- the closed set $\partial^c_0\psi$ and the Borel set $dom~\nabla\psi\times\R^d$ -- the set $S$ is Borel with full measure. Thus $\gamma(Z\cap S)=\gamma(Z)$ for $Z\subset\R^d\times\R^d$. Applied to (\ref{19}), this yields
$$\gamma(U\times V)=\gamma((U\cap T^{-1}(V))\times \R^d)=\mu(U\cap T^{-1}(V))=(id,T)_{\#}\mu(U\times V).$$
$\gamma\in\Gamma(\mu,\nu)$ implies the second equation, Definition~\ref{pushfor} implies the third.
\end{proof}

\begin{remark}
\label{zeromeasure}
Note that by Lemma \ref{newsupdif} and Lemma \ref{supdif}, we have
$$\mu(\{x\in\R^d:T(x)=x\})=0.$$
\end{remark}

\begin{theorem}[geometric representation of optimal solution]
\label{unique}
Assume that $c$ and $l$ satisfy (A1)-(A4), and let $\mu,\nu\in\calp(\R^d)$ be absolutely continuous
with respect to the Lebesgue measure. Then the following hold:
\begin{enumerate}
\item [(a)] There exists an optimal measure $\gamma\in\Gamma(\mu,\nu)$ with $c$-cyclical monotone support;
\item [(b)] For any such $\gamma$, there is a function $\psi:\R^d\rightarrow\R$ which is the $c$-transform of some function on $supp~\gamma$ such that the map
$T:=id-\nabla h^{*}(\nabla\psi)$ pushes $\mu$ forward to $\nu$ and satisfies $\gamma=(id,T)_{\#}\mu$.
\item [(c)] $\gamma:=(S\times id)_{\#}\nu$ for some inverse map $S:\R^d\rightarrow\R^d$. \EEE
\item [(d)]$S(T(x))=x~\mu$-a.e., while $T(S(y))=y~\nu$-a.e.
\end{enumerate}
\end{theorem}
\begin{proof}
For completeness, we sketch the arguments used in \cite{GM96}.

(a) and (b) The existence of an optimal $\gamma$ with $c$-cyclical monotone support is guaranteed by Theorem~\ref{exist}.
Then $supp~\gamma\subset\partial^c_0\psi$. The map $\pi(x,y)=x$ on $\R^d\times\R^d$ pushes $\gamma$ forward to $\mu=\pi_{\#}\gamma$ while
projecting the closed set $\partial^c_0\psi$ to a
$\sigma$-compact set of full measure for $\mu$. From Lemma~\ref{semiconc}, we know that $\psi$ is differentiable $\mu$-a.e. Lemma~\ref{map} shows that
$T(\cdot)$ pushes $\mu$ forward to $\nu$ while $\gamma$ coincides with the measure $(id,T)_{\#}\mu$.
\\
The proofs of (c) and (d) follow the same reasoning as the proofs of (iv) and (v) from Theorem~4.6 in \cite{GM96} and will be omitted.

\end{proof}

\subsection{Uniqueness of the optimal measure}
\begin{lemma}[$c$-superdifferentiability of $c$-transforms]
\label{pointexis}
Let $c$ and $l$ satisfy (A1)-(A4) and let  $V\subset\R^d$ be a closed set. Let $\psi:\R^d\rightarrow\R$ be the $c$-transform of a function on $V$ and suppose that $T:=id-\nabla h^{*}(\nabla\psi)$ can be defined at some $p\in\R^d$ (i.e., $\psi$ is differentiable at $p$ and $\nabla h^{*}$ exists at $\nabla\psi(p)$). Then $\partial_0^c\psi(p)=\{T(p)\}$.
\end{lemma}
\begin{proof}
The proof follows similar arguments as the proof of Proposition~6.1 of \cite{GM96}.
From Lemma~\ref{supdif}, it is clear that $\partial_0^c\psi(p)\subset\{T(p)\}$. Therefore, we only need to prove that $\partial_0^c\psi(p)$ is non-empty. Assume that $T(p)$ is defined for some $p\in\R^d$. By $c$-concavity of $\psi$, there is a sequence ${(y_n,\alpha_n)}_{n=1}^{\infty}\subset A\subset V\times\R$ such that
\begin{equation}
\label{25}
\psi(p)=\lim_{n\rightarrow+\infty}[c(p,y_n)+\alpha_n].
\end{equation}
As is shown below, $(|y_n|)_{n=1}^\infty$ must be bounded. We first assume this bound to complete the proof.
Since $(y_n)_{n=1}^{\infty}$ is bounded, a subsequence must converge to a limit $y_n\rightarrow y$ in
the closed set V. On the other hand, $y\in\partial_0^c\psi(p)$ since for all $x\in\R^d$, (\ref{cconcave}) and (\ref{25}) imply
$$\psi(x)\le\inf_{n\in\bbn} \{c(x,y_n)+\alpha_n\}\le c(x,y)+\psi(p)-c(p,y),$$
with both $\psi(x),\psi(p)>-\infty$, as shown by Lemma~\ref{newsupdif} (2b). Thus, $p\neq y$ and $y\in\partial_0^c\psi(p)$. It remains only to prove that the sequence $(|y_n|)_{n=1}^{\infty}$ is bounded, which means that we can extract a convergent subsequence that converges to a point $y\in Y$.  To produce a contradiction, suppose that a subsequence diverges in a direction $\hat{y}_n\rightarrow\hat{y}$. Then $|p-y_n|$ is bounded away from zero by $\delta>0$. Since for each arbitrary small $x\in\R^d$ we have $|p-y_n-x|>0$, it follows that $\nabla h$ exists at $p-y_n-x$. More precisely, for each $n$ the uniform subdifferentiability in Lemma~\ref{A5} gives
\begin{equation}
\label{subdiff1}
h(p-y_n)\ge h(p-y_n-x)-<x,w_n(x)>+O_\delta(|x|^2),~\mbox{where}~w_n(x)=\frac{{p-y_n-x}}{|p-y_n-x|}l'(|p-y_n-x|)
\end{equation}
for arbitrary small $x\in \R^d$ and $O_\delta(|x|^2)$ independent of $n$.
We re-write
$$<x,w_n(x)>=<x,\frac{p-y_n}{|p-y_n|}>\frac{|p-y_n|}{|p-y_n-x|}l'(|p-y_n-x|)-<x,x>\frac{l'(|p-y_n-x|)}{|p-y_n-x|}.$$
Equation (\ref{subdiff1}) now becomes
\begin{equation}
\label{subdiff2}
h(p-y_n)\ge h(p-y_n-x)-<x,\frac{p-y_n}{|p-y_n|}>\frac{|p-y_n|}{|p-y_n-x|}l'(|p-y_n-x|)-\frac{1}{\delta}|x|^2l'(|p-y_n-x|)+O_\delta(|x|^2).
\end{equation}
Recall now that the derivative of $l$ is negative and increasing and, therefore, $|l'(|p-y_n-x|)|\le |l'((\delta+|x|))$. The sequence $(|y_n|)_{n=1}^\infty$ can only diverge if $|l'(|p-y_n-x|)|$ tends to $|l'(\infty)|:=\inf_{\lambda}|l'(\lambda)|$. Therefore, $\frac{|p-y_n|}{|p-y_n-x|}l'(|p-y_n-x|)\rightarrow l'(\infty)$. Taking a subsequence if necessary ensures that $\frac{p-y_n}{|p-y_n|}$ converge to a limit $w\in \R^d$, with $|w|=1$. Thus $\frac{p-y_n}{|p-y_n-x|}l'(|p-y_n-x|)$ converges to $wl'(\infty)$, which is independent of $x$.
Combining (\ref{subdiff2}) with the definition of a $c$-concave function, this yields
$$\psi(p)\ge\psi(p-x)-<x,w>l'(\infty)+O_\delta(|x|),$$
where the large $n$ limit has been taken using (\ref{25}). By taking $z=-x$ in the above equation, we get
$$\psi(p)+<z,w_1>+O_\delta(|x|)\ge\psi(p+z),~\mbox{where}~w_1:=-wl'(\infty)~\mbox{and}~|w_1|=|l'(\infty)|.$$
Thus $w_1\in\partial^.\psi(p)$. On the other hand, differentiability of $\psi$ at $p$ implies $\partial^.\psi(p)=\{\nabla\psi(p)\}$, whence $w_1=\nabla\psi(p)$. Now $(w_1,p-T(p))\in\partial_.h^*$ follows from the definition of $T(p)$. Assume now that $p=T(p)$. It follows that $(w_1,0)\in\partial_.h^*$. If $h^*$ is non-constant, Lemma~\ref{A4'} gives $(-|y|,0)\in\partial_.l$ -- a result which is obvious when $h^*$ is constant. Thus $(-|y|,0)\in\partial_.l)$ by Lemma \ref{A3} (i). This conclusion contradicts the fact that $l$ is not differentiable at $0$. Therefore $T(p)\neq p$ and $(\nabla h)(p-T(p))= w_1$. Lemma~\ref{A4'} gives $(|p-T(p)|,-|w_1|)\in\partial_.l$. Since $l$ is strictly convex, $|w_1|> |l'(\infty)|$, which produces a contradiction. Therefore, the sequence $(y_n)_{n=1}^{\infty}$ is bounded.
\end{proof}

\begin{theorem}[uniqueness of the optimal map]
\label{geom}
Assume that $c$ and $l$ satisfy (A1)-(A4) and let $\mu,\nu\in \calp(\R^d)$ be such that they are absolutely continuous with respect to the Lebesgue measure.
Then an optimal map $T$ pushing $\mu$ forward to $\nu$ is uniquely determined $\mu$-a.e. by the requirement
that it is of the form $T(x)=x-\nabla h^{*}(\nabla\psi(x))$ for some $c$-concave $\psi$ on $\R^d$.
\end{theorem}
\begin{proof}
For completeness' sake, we will sketch the main idea of the proof, as given
in Theorem 4.4 of \cite{GM96}.

We will prove by contradiction that $T$
is unique. That is, we assume that there exists, in addition to $T$ and
$\psi$, a second c-concave function $\psi'$ for which $T'(x):=x-\nabla
h^*(\nabla\psi'(x))$ pushes $\mu$ forward to $T_{\#}\mu=T'_{\#}\mu=\nu$.
Recall now that $\psi$ and $\psi'$ are $\mu$-a.e. differentiable. $T$ and
$T'$
are defined $\mu$-almost everywhere, and unless they coincide, there
exists some $p\in\R^d$ at which both $\psi$ and $\psi'$ are differentiable
but $T(p)\neq T'(p)$. From this, it is clear that
$\nabla\psi(p)\neq\nabla\psi'(p)$.

Let $U:=\{x\in\R^d|\psi(p)>\psi'(p)\}$. A contradiction will be
derived by showing that the push-forwards $T_{\#}\mu=\nu$ and
$T'_{\#}\mu=\nu$--alleged to coincide--must
differ on $V:=\partial^c\psi(U)$. We will show that
$$\mu(T^{-1}(V))<\mu(U)\le\mu(T'^{-1}(V)).$$
The main ingredient necessary to prove the last equation is the fact that
$\partial_0^c\psi(p)=\{T(p)\}$, which is proved in Lemma \ref{pointexis}.
By using this together with the fact that $\mu$ is absolutely continuous
with respect to the Lebesgue measure, the proof follows via the same arguments
as Theorem 4.4 from \cite{GM96} and will be omitted.
\end{proof}
\subsection{Some general properties of the optimal measure and the optimal cost for equal marginals}

We will investigate in this subsection the case of equal marginals $\mu=\nu$ with common density $\rho$, and
assume throughout that the cost function $c$ satisfies conditions (A1)-(A4). We also introduce the optimal cost
$$E_{OT}^{norm}[\rho]:=\inf_{\gamma\in\Gamma(\rho,\rho)}C[\gamma].$$
$E_{OT}^{norm}[\rho]$ corresponds to the non-normalized functional $E_{OT}$ introduced in Section 3 via
$$E_{OT}^{norm}[\rho]=\frac{1}{{N\choose 2}}E_{OT}[N\rho].$$

For the proof of Theorem \ref{prop} below, we will need the following lemma.
\begin{lemma}
\label{mongeamperre}
For any marginals $\mu$ and $\nu$, and any map $T$ which pushes $\mu$ forward to $\nu$, $\mu,\nu$ and $T$ satisfy the following equation
$$\mu(x)=\nu(T(x))|det~D T(x))|,$$
where $DT$ is the approximate gradient of $T$ (for a definition see for example Definition~10.2 of~\cite{Villani}).

\end{lemma}
\begin{proof}
We have that $\mu(T^{-1}(A))=\nu(A)$ for all Borel sets $A\subseteq\R^d$. Then for any such $A$ we have
\begin{equation}
\label{1}
\int_{y\in T(A)} d\nu(y)=\int_{x\in A}d\mu(x).
\end{equation}
With the change of variables $y=T(x)$ the left-hand side of (\ref{1}) becomes
\begin{equation}
\label{2}
\int_{y\in T(A)} d\nu(y)=\int_{y\in A}|det~D T(x)|\nu(T(x))dx.
\end{equation}
From (\ref{1}) and (\ref{2}), the claim follows.
\end{proof}

\begin{theorem}
\label{prop}
Assume that $\mu=\nu$ with common density $\rho$. Then
\begin{itemize}
\item [(a)] The optimal measure $\gamma_T$ which minimizes $C[\gamma]$ is symmetric, that is
$$\gamma_T(A\times B)=\gamma_T(B\times A)~\mbox{for all Borel}~A,B\in \R^d.$$
\item [(b)] The optimal cost $E_{OT}^{norm}[\rho]$ is strictly convex in $\rho$.
\item [(c)] Let $c(x,y)=1/|x-y|$. Then for all $\alpha>0$ we have the following dilation behaviour
$$E_{OT}^{norm}[\alpha^d\rho(\alpha\cdot)]=\alpha E_{OT}^{norm}[\rho(\cdot)].$$
\end{itemize}
\end{theorem}

\begin{proof}
(a) Recall from Theorem \ref{unique} that $\gamma_T=(id,T)_{\#}\mu$, where $T$ is the optimal map which pushes $\mu$ forward to $\nu$, i.e. $\mu(T^{-1}(A))=\nu(A)$ for all $A\in {\cal B}(\R^d)$. Then
$\gamma_T(x,y)=\mu\left((id,T)^{-1}(x,y)\right)=\delta_{T(x)}(y)\mu(x)$. Using this we get
\begin{equation}
\label{symm}
\gamma_T(A\times B)=\int_{x\in A}\int_{y\in B}\delta_{T(x)}(y)\mu(dx)=\int_{x\in A}\chi_B(T(x))\mu(dx),
\end{equation}
where  $\chi$ denotes the indicator function. We now use Lemma \ref{mongeamperre} and the fact that $\mu=\nu$, then the right hand-side of (\ref{symm}) becomes
\begin{eqnarray*}
\int_{x\in A}\chi_B(T(x))\mu(dx)&=&\int_{x\in A}\chi_B(T(x))\mu(T(x))|det~D T(x)|dx=\int_ {y\in T^{-1}(A)}\chi_B(y)\mu(y)dy\\
&=&\int_{y\in B}\chi_{T^{-1}(A)}(y)\mu(y)dy=\int_{y\in B}\chi_{A}(T(y))\mu(y)dy\\
&=&\gamma_T(B\times A).
\end{eqnarray*}
(b)  The convexity is immediate from the definition of $E_{OT}^{norm}[\rho]$, and strict convexity follows from uniqueness.\\
(c)  Fix any $\alpha>0$. Then
\begin{eqnarray*}
E_{OT}^{norm}[\alpha^d\rho(\alpha\cdot)] &=&\inf_{\tilde\gamma\in\Gamma(\alpha^d\rho(\alpha\cdot), \alpha^d\rho(\alpha\cdot))}\int \frac{1}{|x-y|}\tilde\gamma(x,y)dx dy\\
&=&\inf_{\alpha^{2d}\gamma(\alpha x,\alpha y):\gamma\in\Gamma(\rho,\rho) }\int \frac{1}{|x-y|}\alpha^{2d}\gamma(\alpha x,\alpha y)dx dy\\
&=&\inf_{\gamma\in\Gamma(\rho,\rho) }\int \frac{1}{|x-y|}\alpha^{2d}\gamma(\alpha x,\alpha y)dx dy\\
&=& \inf_{\gamma\in\Gamma(\rho,\rho) }\alpha\int \frac{1}{|x-y|}\gamma(x',y')dx' dy'\\
&=&\alpha E_{OT}^{norm}[\rho(\cdot)],
\end{eqnarray*}
where for the second equality we used the fact that for $\gamma\in\Gamma(\rho,\rho)$ we have $\int\alpha^{2d}\gamma(\alpha x,\alpha y)dy=\alpha^d\rho(\alpha x)$ and for the penultimate equality we used the change of variables $x'=\alpha^d x, y'=\alpha^d y$.
\end{proof}

\begin{remark}
(0) Trivially, the statements in Theorem \ref{prop} also hold for the non-normalized functional $E_{OT}$.

(1) Recall from Lemma \ref{DFT1} that the exact energy interaction is of form
$$V_{ee}[\Psi] = \int_{\R^6} \frac{1}{|x-y|}\rho_2(x,y)\, dx\, dy,$$
with the exact pair density $\rho_2$ being symmetric due to the antisymmetry condition on the underlying $\Psi$ in (\ref{class}). Property (a) of Theorem~\ref{prop} shows
that the approximate interaction energy
$$E_{OT}[\rho]=\int_{\R^6}\frac{\rho_2^{opt}(x,y)}{|x-y|}\,dx\,dy,$$
is of the same form, with the arising $\rho_2^{opt}$ being automatically symmetric as a consequence of optimality coupled with the weaker symmetry condition
that $\rho_2^{opt}$ has equal marginals.

(2) Property (c) of Theorem~\ref{prop} is a scaling property of the exact electron-electron energy $V_{ee}$ not shared by many approximate density functionals used in the physics
literature, such as the local density approximation (\ref{LDA}).

(3) The dilation behaviour of $E_{OT}$ equals that of the exact $V_{ee}$, as well as that of approximations like (\ref{LDA}).
\end{remark}

\section{Explicit example - equal radially symmetric marginals}

As in the last subsection, we continue to investigate the case of equal marginals $\mu=\nu$ with common density $\rho \, : \, \R^d\to [0,\infty)$.
Moreover, we assume that $\rho(x)>0$ for all $x\in~supp~\mu$. We will also assume throughout that the cost function $c$ satisfies conditions (A1)-(A4).

Throughout this section, for any dimension $d\in\bbn$ we will denote the optimal map by $T^{(d)}$.
In subsection~4.1 we will explicitly compute $T^{(1)}$, and in subsection~4.2 we use the one-dimensional analysis to explicitly compute $T^{(d)}$ when $\rho$ is radially
symmetric, that is to say when $\rho(x)=\lambda(|x|)$ for all $x\in\R^d$ and some function $\lambda$.

As turns out, in the above situations the optimal map is {\it universal} with respect to all cost functions satisfying (A1)-(A4),
but the fact that $c(x,y)$ decreases with the distance $|x-y|$ is essential.

\subsection{Explicit solution for equal marginals in one dimension}

Let $\mu=\nu\in\calp(\R)$ be equal marginals on $\R$.
Moreover, we define $I:=supp~\mu$.
Recall from Theorem~\ref{3.2} that the unique optimal measure has $c$-cyclically monotone support. This will help us to characterize the optimal
map $T^{(1)}$ in the following lemmas.

\begin{lemma}
\label{exclus}
Let $(x_1,y_1)$ and $(x_2,y_2)$ be two points in the support of the optimal map $\gamma$, that is $y_1=T^{(1)}(x_1)$ and $y_2=T^{(1)}(x_2)$. The possible configurations (not counting the symmetries between $(x_1,y_1)$ and $(x_2,y_2)$) are: $x_1<x_2\le y_1\le y_2$, $x_1\le y_2\le y_1\le x_2$, $y_1\le x_2<x_1\le y_2$, $y_1\le y_2\le x_1<x_2$, $x_1\le y_2\le x_2\le y_1$, $y_1\le x_2\le y_2\le x_1$, $x_1\le y_1\le x_2\le y_2$ and $y_1\le x_1\le y_2\le x_2$ (see also Figure~\ref{exclud} for the excluded configurations).
\begin{figure}
\begin{center}
\resizebox{!}{4cm}{\includegraphics{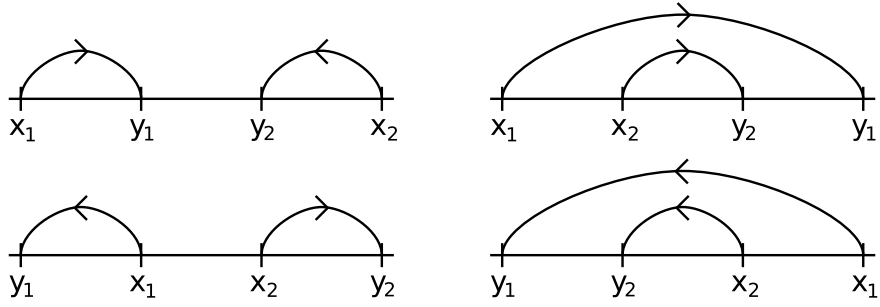}}
\caption{\label{exclud} Configurations excluded by Lemma~\ref{exclus}}.
\end{center}
\end{figure}
\end{lemma}

\begin{proof}
If $(x_1,y_1)$ and $(x_2,y_2)$ are two points in the support of the optimal map $\gamma$, then
$$c(x_1,y_1)+c(x_2,y_2)\le c(x_1,y_2)+c(x_2,y_1).$$
Let us consider the excluded cases one by one.
\begin{enumerate}
\item [(i)] $x_1\le y_1<y_2\le x_2$.

Then, due to the fact that $l$ is strictly decreasing, it follows that
$$c(x_1,y_1)+c(x_2,y_2)> c(x_1,y_2)+c(x_2,y_1),$$
which contradicts the $c$-cyclically monotonicity property of the optimal solution.
\item [(ii)] $y_1\le x_1<x_2\le y_2$.

Similar to (i).
\item [(iii)] $x_1<x_2\le y_2<y_1$.

We have $y_2-x_2<y_2-x_1< y_1-x_1$. Therefore,
$y_2-x_1=t(y_2-x_2)+(1-t)(y_1-x_1)$ and $y_1-x_2=(1-t)(y_2-x_2)+t(y_1-x_1)$, where $t\in [0,1]$. Thus, using the strict convexity of $h$, we have
$$c(x_2,y_1)+c(x_1,y_2)<tc(x_2,y_2)+(1-t)c(x_1,y_1)+(1-t)c(x_2,y_2)+tc(x_1,y_1)=c(x_1,y_1)+c(x_2,y_2),$$
which contradicts the $c$-cyclically monotonicity property of the optimal solution.
\item [(iv)] $y_1<y_2\le x_2<x_1$

Similar to (iii).
\end{enumerate}
\end{proof}

\begin{figure}
\begin{center}
\resizebox{!}{4cm}{\includegraphics{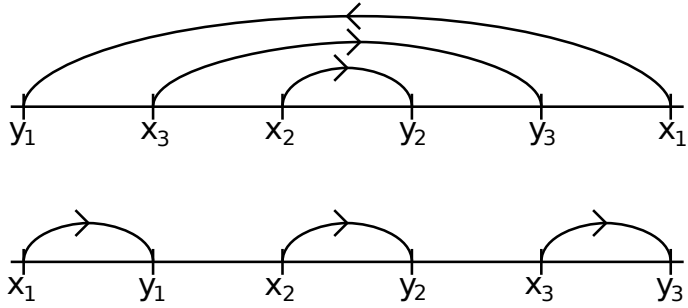}}
\caption{\label{newexclud} Example of configurations excluded by Lemma~\ref{config0}}.
\end{center}
\end{figure}

\begin{lemma}
\label{config0}
Assume $(x_1,y_1), (x_2,y_2)\in supp~\gamma$ are such that one of the following four configurations holds: $x_1<y_2<x_2<y_1$ or $y_1<x_2<y_2<x_1$ or $x_1<y_1<x_2<y_2$ or $y_1<x_1<y_2<x_2$. Then, if $(x_3,y_3)$ is another point in the $supp~\gamma$, none of the following configurations are possible: $x_i<x_k<y_j<x_j<y_k<y_i$, $x_i<y_k<y_j<x_j<x_k<y_i$, $y_i<x_k<x_j<y_j<y_k<x_i$, $y_i<y_k<x_j<y_j<x_k<x_i$, $x_i<y_i<x_j<y_j<x_k<y_k$ and $y_i<x_i<y_j<x_j<y_k<x_k$, where $i,j,k\in\{1,2,3\}$ (see also Figure \ref{newexclud}).
\end{lemma}

\begin{proof}
Note that the first $4$ configurations are immediately excluded by Lemma~\ref{exclus}. Let us focus on the penultimate configuration. From the $c$-cyclically
monotonicity property, we have that
$$c(x_1,y_1)+c(x_2,y_2)+c(x_3,y_3)\le c(x_j,y_k)+c(x_i,y_j)+c(x_k,y_i).$$
But due to the fact that $h$ is strictly decreasing, we have that $c(x_i,y_i)>c(x_i,y_j)$, $c(x_j,y_j)>c(x_k,y_i)$ and $c(x_k,y_k)>c(x_j,y_k)$, which gives rise to a contradiction. The last configuration can be dealt with in a similar way.
\end{proof}

\begin{remark}
\label{config1}
Note that by Lemma~\ref{exclus}, if $(x_1,y_1),(x_2,y_2)$ and $(x_3,y_3)\in supp~\gamma$, with $x_1<x_2<x_3$ the following configurations are also not possible: $x_1<y_2<y_1<x_2<x_3<y_3$ and $x_1<y_2<y_1<y_3<x_2<x_3$. Similarly, the configurations $y_1<x_1<x_2<y_3<y_2<x_3$ and $x_1<x_2<y_1<y_3<y_2<x_3$ are not possible.
\end{remark}

\begin{remark}
\label{prob1config}
From Lemma \ref{exclus}, Lemma \ref{config0} and Remark \ref{config1}, it follows that the configurations $\mu$-a.e. possible are of form: $x_1<x_2<y_1<y_2$, $x_1<y_2<y_1<x_2$, $y_1<y_2<x_1<x_2$, $y_1<x_2<x_1<y_2$, $x_1\le y_2<x_2\le y_1$ and $y_1\le x_2<y_2\le x_1$ (see also Figure \ref{PC}).
\end{remark}

\begin{figure}[t]
\begin{center}
\resizebox{!}{4cm}{\includegraphics{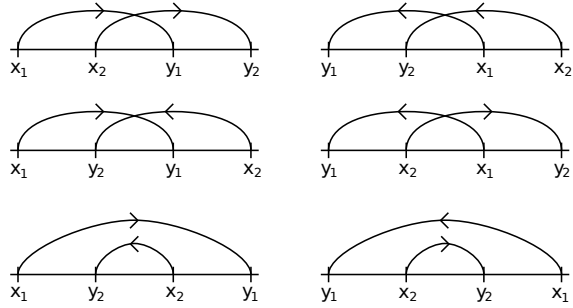}}
\caption{\label{PC}Possible configurations by Remark \ref{prob1config}}
\end{center}
\end{figure}

\begin{defn}
We say that $T^{(1)}$ has no points of decrease on $A\subseteq I$ if $\mu(\{x\in A:\exists x'\in A, x'>x,T^{(1)}(x)>T^{(1)} (x')\})=0$. We say that $T^{(1)}$ has points of decrease on $A$ with positive measure if $\mu(\{x\in A:\exists x'\in A, x'>x,T^{(1)}(x)>T^{(1)} (x')\})>0$. We say that $T^{(1)}$ is $\mu$-a.e. decreasing on $A\subseteq I$ if $\mu(\{x\in A:\exists x'\in A, x'>x,T^{(1)}(x)>T^{(1)} (x')\})=1$. We define similarly for $T^{(1)}$ the notions of points of increase on $A$ and $\mu$-increasing on $A$.

Note that we can assume that the set $B$ of such $x'$ above such that $T^{(1)}$ has points of decrease (respectively points of increase) on $A$ with positive measure, is also a set of positive measure. Otherwise, if the set of such $x'$ is of $\mu$-measure zero, we may consider the set $A\setminus B$, on which $T^{(1)}$ is $\mu$-a.e. decreasing (respectively $\mu$-a.e. increasing).
\end{defn}

\begin{lemma}
\label{nonincreas}
The map $T^{(1)}$ cannot be $\mu$-a.e. decreasing on any subset $A\subseteq I$.
\end{lemma}

\begin{proof}
Asssume that  $T^{(1)}$ is $\mu$-a.e. decreasing on a subset $A\subseteq I$. Let $(x_i,y_i)_{i=1}^3=(x_i, T^{(1)}(x_i))_{i=1}^3$,
where $x_i\in A$ for $i\in\{1,2,3\}$. Recall now from Remark \ref{prob1config} the possible configurations by which $T^{(1)}$ is decreasing.

Assume first that $x_1<y_2<x_2<y_1$. Then the possibilities for $(x_3,y_3)$ such that $T^{(1)}$ is strictly decreasing, are: $x_3<x_1<y_2<x_2<y_1<y_3$,
$x_1<x_3<y_2<y_3<x_2<y_1$, $x_1<y_2<y_3<x_3<x_2<y_1$, $x_1<y_2<x_3<y_3<x_2\le y_1 $, $y_3<x_1<y_2<x_2<x_3<y_1$, $x_1<y_3<y_2<x_2<x_3<y_1 $,
$y_3<x_1<y_2<x_2<y_1<x_3$ and $x_1<y_3<y_2<x_2<y_1<x_3$. In view of Lemmas \ref{exclus} and \ref{config0}, each these possibilities can only happen on a set of $\mu$-a.e. measure $0$. The case $y_1<x_2<y_2<x_1 $ can be treated similarly.

Assume next that $x_1<y_2<y_1<x_2$. Then the possibilities for $(x_3,y_3)$ such that $T^{(1)}$ is strictly decreasing, are: $x_3<x_1<y_2<y_1<y_3<x_2$,
$x_3<x_1<y_2<y_1<x_2<y_3$, $x_1<x_3<y_2<y_3<y_1<x_2$, $x_1<y_2<x_3<y_3<y_1<x_2$, $x_1<y_2<y_3<x_3<y_1<x_2$, $x_1<y_2<y_3<y_1<x_3<x_2$, $y_3<x_1<y_2<y_1<x_2<x_3$ and $x_1<y_3<y_2<y_1<x_2<x_3$. In view of Lemmas \ref{exclus} and \ref{config0}, each these possibilities can only happen on a set of $\mu$-a.e. measure $0$. The case $y_1<x_2<x_1<y_2$ can be treated in a similar way.
\end{proof}

\begin{remark}
\label{finite}
We have $supp~\gamma=I\times I$;
 in particular, $Im(T^{(1)})=I$, where we denoted by $Im(T^{(1)})$ the image of $T^{(1)}$.
\end{remark}
\begin{proof}
Note first that $\gamma(I\times I)=\gamma(I\times\R)=\mu(I)=1$. Let us now assume that $supp~\gamma=I\times J$, with $I,J\subseteq\R$ and $J\subset I$, with $\mu(I\setminus J)>0$.
Then
$1=\gamma(I\times J)=\gamma(I\times \R)=\mu(I)=\gamma(\R\times J)=\mu(J)$, which contradicts the definition of the support $I$ of the marginals.
\end{proof}

For the proof of the next theorem, we will use the results of Theorems \ref{unique} and \ref{geom}; in particular, we will use the properties of the map $T^{(1)}$, as given in those two theorems.

\begin{theorem}
\label{tdescrip}
Let $\alpha,\beta\in\R$ with $\alpha<\beta$ and let $I=[\alpha,\beta]$. There exists $a\in I$ such that $T^{(1)}$ is $\mu$-a.e. increasing on $[\alpha,a)$ and $\mu$-a.e. increasing on $(a,\beta]$, with $T^{(1)}(\alpha)=T^{(1)}(\beta)=a$, $T^{(1)}(a_-)=\beta$ and $T^{(1)}(a_+)=\alpha$, with discontinuity at $a$. Except on a set of $\mu$-measure zero, we have
\begin{enumerate}
\item [(a)] For all $x_1,x_1'\in (\alpha,a)$ with $x_1<x_1'$, we have $x_1<x_1'<T^{(1)}(x_1)< T^{(1)}(x_1')$, with $T^{(1)}(x)\in (a,\beta)$;
\item [(b)] For all $x_2,x_2'\in (a,\beta)$ with $x_2<x_2'$, we have $T^{(1)}(x_2)< T^{(1)}(x_2')<x_2<x_2'$, with $T^{(1)}(x)\in (\alpha,a)$;
\item [(c)] For every interval $(l_1l_2)\subseteq I$, we have $T((l_1,l_2))=(r_1,r_2)\subseteq I$.
\end{enumerate}
Moreover, $a$ is such that $\mu([\alpha,a])=\mu([a,\beta])=\frac{1}{2}$ (see also Figure \ref{OM}).
\begin{figure}
\begin{center}
\resizebox{15cm}{!}{\includegraphics{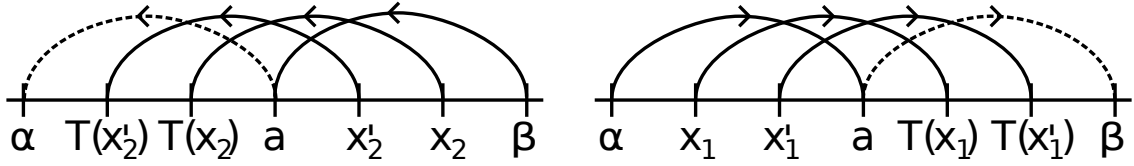}}
\caption{\label{OM} Optimal map configurations}
\end{center}
\end{figure}
\end{theorem}

\begin{proof}
Recall first from Remark \ref{zeromeasure} that
$$\mu(\{x\in\R^d:T^{(1)}(x)=x\})=0.$$
Recall also from Theorem \ref{unique} that $T$ is a bijective map; in particular, $\mu(x\in I:\exists y\in I\setminus \{x\}, T^{(1)}(x)=T^{(1)}(y))=0$.

\textbf{Step 1. $\boldsymbol{T^{(1)}}$ cannot be $\boldsymbol{\mu}$-a.e. increasing on $I$:}

Assume that $T^{(1)}$ is $\mu$-a.e. increasing on $I$. Then $T^{(1)}$ is $\mu$-a.e. strictly increasing and it can only be increasing as described in Remark \ref{prob1config}. Suppose first that for $\mu$-a.e. all $x_1,x_2\in I$, we have $x_1<x_2<T^{(1)}(x_1)<T^{(1)}(x_2)$. Then for each $x\in [\alpha,\beta]$, we have two possibilities: $T^{(1)}(\alpha,x)=(\alpha,T^{(1)}(x))$ or $T^{(1)}(\alpha,x)\subseteq (c,T^{(1)}(x))$, with $\alpha<c$. If $T^{(1)}(\alpha,x)=(\alpha,T^{(1)}(x))$, using $\mu=\mu\circ (T^{(1)})^{-1}$ and the fact that $\rho>0$, we get that $T^{(1)}(x)=x~\mu$-a.e., which contradicts Remark~\ref{zeromeasure}. If $T^{(1)}(\alpha,x)\subseteq (c,T^{(1)}(x))$, then $T^{(1)}(\alpha,\beta)\subseteq (c,T^{(1)}(\beta))$. Using again $\mu=\mu\circ (T^{(1)})^{-1}$, we get that $\mu(\alpha,\beta)=1\le\mu(c,T^{(1)}(\beta))$, where $\alpha<c$, which would contradict the definition of $supp~\mu=[\alpha,\beta]$. The case with $T^{(1)}(x_1)<T^{(1)}(x_2)<x_1<x_2$ for $\mu$-a.e. all $x_1, x_2\in I$, can be discounted the same way.

\textbf{Step 2. $\boldsymbol{T^{(1)}}$ cannot have both points of increase and points of decrease on every interval $\boldsymbol{A\subseteq I}$ with positive measure:}

Assume that there exists an interval $A\subseteq I$, such that $T^{(1)}$ has both points of increase and points of decrease on $A$ with positive measure. Take an arbitray point $x_1\in A$ such that $\exists x'_1\in A, x'_1>x_1$ and $T^{(1)}(x_1')>T^{(1)}(x_1)$. By assumption, for a subset of such $x_1$ (and for a subset of $x_1'$) in $A$ of positive measure, $\exists x_2, x_2'\in (x_1,x_1')$ with $x_2'>x_2$ and  $T^{(1)}(x_2')<T^{(1)}(x_2)$. Assume now that $x_1<x_1'<T^{(1)}(x_1)<T^{(1)}(x_1')$. Due to the $\mu$-a.e. possible configurations as given by Remark \ref{prob1config}, $x_2'$  is such that $T^{(1)}(x_2')<x_2'$. Then the $c$-cyclical monotonicity of the optimal support fails for $(x_1',T^{(1)}(x_1'))$ and $(x_2',T^{(1)}(x_2'))$. If we denote by $B$ the set of such $x_1'$, we have $\gamma(B,T^{(1)}(B))=\mu(B)>0$, which contradicts the assumption on the optimal support. The case with  $T^{(1)}(x_1)<T^{(1)}(x_1')<x_1<x_1'$ can be treated similarly so its proof will be omitted.

\textbf{Step 3. There exists $\boldsymbol{(\alpha,\alpha_1),(\beta_1,\beta)\subset I}$, with $\boldsymbol{\alpha_1\le\beta_1}$, such that $\boldsymbol{\mu}$-a.e. for all $\boldsymbol{x_1\in (\alpha,\alpha_1)}$ we have $\boldsymbol{T^{(1)}(x_1)>x_1}$ and $\boldsymbol{\mu}$-a.e. for all $\boldsymbol{x_2\in (\beta_1,\beta)}$ we have $\boldsymbol{T^{(1)}(x_2)< x_2}$:}

Recall first the possible configurations, as given by Remark~\ref{prob1config}. Note now that in view of Step 1, Step 2 and of Lemma \ref{nonincreas}, there exists $(\alpha,\alpha_1), (\beta_1,\beta)\subset I$, with $\alpha_1\le\beta_1$ on which $T^{(1)}$ is an increasing function. It remains to show that the optimal map can only be such that $\mu$-a.e. for all $x_1\in (\alpha,\alpha_1)$, we have $T^{(1)}(x_1)> x_1$ and $\mu$-a.e. for all $x_2\in (\beta_1,\beta)$, we have $T^{(1)}(x_2)<x_2$.

Let us consider the alternatives one by one. Suppose to begin with that $\mu(x_1\in (\alpha,\alpha_1):T^{(1)}(x_1)<x_1)>0$ and $\mu(x_2\in (\beta_1,\beta):x_2<T^{(1)}(x_2))>0$. By Lemma \ref{exclus}, if $x_1\in (\alpha,\alpha_1)$ with $T^{(1)}(x_1)<x_1$ and $x_2\in (\beta_1,\beta)$ with $x_2<T^{(1)}(x_2))$, $c$-cyclical monotonicity of the support fails for $(x_1,T^{(1)}(x_1))$ and $(x_2,T^{(1)}(x_2))$. Assume next that $\mu(x_1\in (\alpha,\alpha_1):T^{(1)}(x_1)<x_1)>0$ and $\mu(x_2\in (\beta_1,\beta):T^{(1)}(x_2)<x_2)>0$. In view of Step $1$, of Lemma \ref{nonincreas} and of Lemma \ref{finite}, there exists then some subset $(\alpha_1',\beta_1')\subseteq (\alpha_1,\beta_1)$ such that with positive measure, $T^{(1)}$ has both points of increase and points of decrease on $(\alpha_1',\beta_1')$. But this contradicts the conclusion of Step 2 and therefore our assumption has to be wrong. The case with $\mu(x_1\in (\alpha,\alpha_1): x_1<T^{(1)}(x_1))>0$ and $\mu(x_2\in (\beta_1,\beta):x_2<T^{(1)}(x_2))>0$ can be reasoned similarly, so its proof will be omitted.

\textbf{Step 4. $\boldsymbol{T^{(1)}(\alpha)\ge T^{(1)}(\beta)}$:}

Assume that $T^{(1)}(\alpha)<T^{(1)}(\beta)$. By Step $3$, there exist $(\alpha,\alpha')\subset (\alpha,T^{(1)}(\alpha))$ and $(\beta',\beta)\subset (T^{(1)}(\beta),\beta)$ on which $T^{(1)}$ is as described in (a) and (b). Then, if $x_1\in (\alpha,\alpha')$ and $x_2\in (\beta',\beta)$,  $c$-cyclical monotonicity of $supp~\gamma$ would fail for $(x_1,T^{(1)}(x_1))$ and $(x_2,T^{(1)}(x_2))$, as shown in Lemma \ref{exclus} (i). Therefore, $T^{(1)}(\alpha)\ge T^{(1)}(\beta)$.

\textbf{Step 5. There exists $\boldsymbol{b\in I}$ such that $\boldsymbol{T^{(1)}}$ is as described in (a) and (b), for $\boldsymbol{(\alpha,b)\subset I}$ and for $\boldsymbol{(b,\beta)\subset I}$, respectively:}

Note that by Lemma \ref{nonincreas}, Step 1 and Step $2$, $T^{(1)}$ has to be $\mu$-a.e. increasing on a certain number of sub-intervals of $(\alpha_1,\beta_1)$. On any such sub-intervals, either $T^{(1)}(x)<x$ $\mu$-a.e. or $T^{(1)}(x)>x$ $\mu$-a.e. In both these cases, due to the form of $T^{(1)}$ on $(\alpha,\alpha_1)$ and $(\alpha,\alpha_1)$, as proved in Step $3$, the $c$-cyclical monotonicity of the support would fail on a set of positive measure unless $\mu((\alpha_1,\beta_1))=0$.

\textbf{Step 6. For every interval $\boldsymbol{(l_1l_2)\subseteq I}$, we have $\boldsymbol{T((l_1,l_2))=(r_1,r_2)\subseteq I}$:}

This is a simple consequence of Step $5$ and of Remark \ref{finite}.

\textbf{Step 7. $\boldsymbol{b=T^{(1)}(\alpha)=T^{(1)}(\beta)}$, $\boldsymbol{T^{(1)}(b_-)=\beta}$ {and} $\boldsymbol{T^{(1)}(b_+)=\alpha}$:}

Note first that $T^{(1)}(b_-)=\beta$ and $T^{(1)}(b_+)=\alpha$ or else $(\alpha,a)$ and $(a,\beta)$ will be mapped into a smaller interval than $I$, which contradicts Remark \ref{finite}. By Step 4, we have $T^{(1)}(\alpha)\ge T^{(1)}(\beta)$. It remains to prove that $b=T^{(1)}(\alpha)=T^{(1)}(\beta)$. If this does not hold, the alternatives are: $T^{(1)}(\alpha)\ge b>T^{(1)}(\beta)$, $T^{(1)}(\alpha)>T^{(1)}(\beta)\ge b$ and $T^{(1)}(\beta)<T^{(1)}(\alpha)\le b$. We will only show the reasoning for the case $T^{(1)}(\alpha)\ge b>T^{(1)}(\beta)$, as the other two possibilities can be dealt with in a
similar way. In this first case, we map $(\alpha,b)$ to $(T^{(1)}(\alpha),\beta)$ and $(b,\beta)$ to $(\alpha, T^{(1)}(\beta))$. Therefore, $supp~\gamma\subset I\times ((\alpha, T^{(1)}(\beta))\cup (T^{(1)}(\alpha),\beta)\neq I\times I$, which contradicts Remark \ref{finite}.

\textbf{Step 8. $\boldsymbol{\mu([\alpha,a])=\mu([a,\beta])=\frac{1}{2}}$:}

$(\alpha,a)$ is mapped into $(a,\beta)$ and $(a,\beta)$ is mapped into $(\alpha,a)$. Therefore
$$\gamma((\alpha,a)\times (a,\beta))+\gamma((a,\beta)\times (\alpha,a))=1.$$
But
$$\gamma((\alpha,a)\times (a,\beta))=\gamma((\alpha,a)\times I)=\mu((\alpha,a)), $$
as $\gamma((\alpha,a)\times (\alpha,a))=0$. Similarly,
$$\gamma((a,\beta)\times (\alpha,a))=\mu((\alpha,a)).$$
It follows that $2\mu((\alpha,a))=1$ or $\mu((\alpha,a))=\frac{1}{2}$.
\end{proof}

\begin{theorem}\label{onedimrep}
Assume that $\mu=\nu$ with density $\rho(x)>0$ on $I=[\alpha,\beta]$, where $\alpha,\beta\in\R\cup\{\pm\infty\}$. Let $\mu_1(x):=\mu((\alpha,x))$, $\bar{\mu}_1:=\mu((x,a))$ for $x\in (\alpha,a)$, $\mu_2(x):=\mu((x,\beta))$ and $\bar{\mu}_2:=\mu((a,x))$, for $x\in (a,\beta)$. If $x\in (\alpha,a)$, we have $T^{(1)}(x)={\bar{\mu}_2}^{-1}(\mu_1(x))$ and if $x\in (a,\beta)$, we have $T^{(1)}(x)={\bar{\mu}_1}^{-1}(\mu_2(x))$.
\end{theorem}

\begin{proof}
We will use the fact that $\mu\circ T^{(1)}=\mu$ to find $T^{(1)}$. Let $x\in (\alpha,a)$. Then due to the properties of $T^{(1)}$ from Theorem~\ref{tdescrip}, it follows that $T^{(1)}((\alpha,x)=(a,T^{(1)}(x))$. Therefore,
$$\mu_1(x)=\mu((\alpha,x))=\mu((a,T^{(1)}(x))=\bar{\mu}_2(T^{(1)}(x)).$$
We know that $\rho(x)>0$. Due to the fact that $T^{(1)}(x)$ is increasing on $(\alpha,a)$ and with $T^{(1)}(a_-)=\beta$, $\mu_2(T^{(1)}(x)$ is a a strictly increasing function. We can take inverses and have
$$T^{(1)}(x)={\bar{\mu}_2}^{-1}(\mu_1(x)).$$
A similar reasoning holds for $x\in (a,\beta)$.
\end{proof}

\subsection{Equal radially symmetric marginals in dimension $d$}

We assume in this subsection that the marginals $\mu$ and $\nu$ are radially symmetric and in $\calp(\R^d)$.
As before, we suppose that the cost function is given by $c(x,y)=\ell(|x-y|)\ge 0$, with $c$ and $\ell$
satisfying (A1)--(A4).

\begin{theorem}
\label{bigdim}
Suppose that $\mu=\nu$, with common density $\rho(x)=\lambda(|x|)$ for all $x\in~supp~\mu$.
Moreover, we assume that $\rho(x)>0$ for all $x\in~supp~\mu$. Then the optimal
transport map $T^{(d)}$ has to be radially symmetric itself, that is
\begin{equation}
\label{trepres}
   T^{(d)}(x)=g(|x|)\frac{x}{|x|},\quad x\in\R^d,
\end{equation}
for some function $g\, : \, [0,\infty)\to\R$. Moreover $g\le 0$, and $g$
is an increasing function with $g(0_+)=-\infty$ and $g(+\infty)=0$.
\end{theorem}

\begin{proof}
{\bf Step 1. $\boldsymbol{T(Rx) = RT(x)}$ for all $\boldsymbol{R\in O(d)}$ and all $\boldsymbol{x\in\,supp~\mu}$:}
(Here $O(d)$ denotes the group of orthogonal $d\times d$ matrices):

 Let $\gamma_T$ be a minimizer of $C$ on
$\Gamma(\rho,\rho)$. Then $(R\times R)_\sharp\gamma$ is also a minimizer, for any $R\in O(n)$, since it belongs to $\Gamma(\rho,\rho)$ by the radial symmetry of $\rho$, and has the same cost $C$ as $\gamma_T$ by the
invariance of the cost function $c(x,y)$ under $(x,y)\mapsto (R^{-1}x,R^{-1}y)$. Hence by uniqueness,
$\gamma_T=(R\times R)_\sharp\gamma_T$. But an elementary calculation shows that the latter is equivalent to
$T(z)=R T(R^{-1}z)$ for all $z\in supp\,\mu$.
Left, respectively right hand side, evaluated on a set $A\times B$ give
$\int \chi_A(x)\chi_B(T(x))\rho(x)dx$ respectively $\int \chi_A(Rx)\chi_B(RT(x))\rho(x)dx$. A change of variables
together with the radial symmetry of $\rho$ shows that the latter expression equals
$\int \chi_A(z)\chi_B(RT(R^{-1}z))\rho(z)dz$. Comparing with the former expression yields the assertion.

{\bf Step 2. $\boldsymbol{T}$ is radial and direction reversing, i.e. $\boldsymbol{T(x)=g(|x|)\frac{x}{|x|}}$ for some $\boldsymbol{g\le 0}$:}

Let $e_1$ be a fixed unit vector in $\R^d$ and $r>0$. By Step 1, for all $R\in O(n)$ we have
\begin{equation}\label{Trep}
   T(R r e_1)=R T(re_1)=R f(r) v(r)\mbox{ with }f(r):=|T(re_1)|\mbox{ and }v(r):=T(re_1)/|T(re_1)|.
\end{equation}
Hence
$$
   I[T] = \int \ell(|x-T(x)|)\rho(x)\, dx = \int \ell \bigl(|x|e_1-f(|x|)v(|x|)\bigr)\, \rho(x)\, dx.
$$
But $\ell$ is by assumption strictly decreasing and $|re_1-f(r)v(r)|$ is maximized among unit vectors $v(r)$
if and only if $v(r)=-e_1$. Hence, since $T$ minimizes $I$, $v(r)=-e_1$. Substituting into
\eqref{Trep} yields the assertion, with $g(r)=-f(r)=-|T(re_1)|$.

{\bf Step 3. $\boldsymbol{g}$ solves a one-dimensional mass transportation problem:}

For any Borel map $T$ on $\R^d$, abbreviate $I_\mu[T]:=\int \ell(|x-T(x)|)\rho(x)dx$
(Monge functional with map $T$ and equal marginals $\rho(x)=\lambda(|x|)$). If $T$ is a radial map, i.e. of form
$T(x)=g(|x|)\frac{x}{|x|}$ for some Borel $g\, : [0,\infty)\to\R$, and $\tilde{g}$ denotes the antisymmetric
extension of $g$ to $\R$, such that, in particular,
$T(x,0,0)=(\tilde{g}(x),0,0)$ for all $x\in\R$, then using polar coordinates (with $|S^{d-1}|$ denoting the Hausdorff measure
of the unit sphere in $\R^d$)
\begin{eqnarray*}
 I_\rho[T]&=&\int_{r=0}^\infty \ell(|r-g(r)|) |S^{d-1}|r^{d-1}\lambda(r) dr \\
          &=&\int_{s=-\infty}^\infty \ell(|s-{\tilde{g}}(s)|)
          \underbrace{\mbox{$\frac12$}|S^{d-1}||s|^{d-1}\lambda(|s|)}_{=:\rho_1(s)} ds = I_{\rho_1}[T].
\end{eqnarray*}
Hence the $d$-dimensional Monge problem of minimizing $I_\rho$ over radial maps
is equivalent to the one-dimensional Monge problem of minimizing $I_{\rho_1}$ over antisymmetric maps,
and -- because of Step 1 (with d=1 and $R=-I$) -- to the one-dimensional Monge problem of minimizing $I_{\rho_1}$
over arbitrary maps. It follows that the function $g$ in \eqref{trepres}, antisymmetrically extended to $\R$,
is a minimizer of $I_{\rho_1}$. The asserted properties of $g$ now follow immediately from
Theorem \ref{tdescrip} and the fact that $\rho_1$ (being symmetric) has median $0$.
\end{proof}

\begin{corollary}
\label{cor1}
Suppose that $\mu=\nu$ are as in Theorem \ref{bigdim}. Let $t\in(0,\infty)$ and denote by
$$F_1(t)=|S^{d-1}|\int_0^t\lambda(s)s^{d-1}ds~\mbox{and}~F_2(-t)=|S^{d-1}|\int_{t}^{\infty}\lambda(s)s^{d-1}ds.$$
Then the function $g$ in \eqref{trepres} is given by
$$g(t)=F_2^{-1}(F_1(t)).$$
\end{corollary}
\begin{proof}
We have already shown that $g$, antisymmetrically extended to $\R$,
minimizes the one-dimensional functional $I_{\rho_1}$, with $\rho_1$ as in Step 3 above.
The assertion is now a direct consequence of the representation formula given in Theorem \ref{onedimrep}.
\end{proof}

\bexam (exponential radially symmetric distribution)
\label{Example1}
Assume now that $\rho(x)=\frac{1}{Z}e^{-|x|}$ for $x=(x_1,_2,x_3)\in\R^3$, where $Z$ is the normalizing constant, that is, $Z=\int e^{-|x|} dx_1dx_2 dx_3$. Then for $t\in (0,\infty)$, we have
$$ F_1(t)=1-\left(1+t+\frac{t^2}{2}\right)e^{-t},~F_2(-t)=e^{-t}\left(1+t+\frac{t^2}{2}\right)~\mbox{and}~g(t)=F^{-1}_2(F_1(t)).$$
\eexam

\begin{figure}
~\hfill
\begin{minipage}[t]{0.4\textwidth}
\resizebox{!}{4cm}{\includegraphics{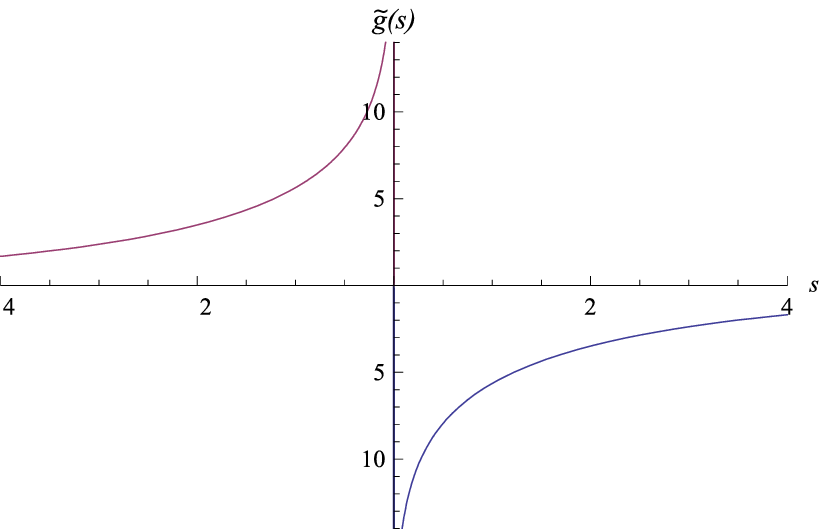}}
\caption{\label{jAs1} Optimal transport map $T$ for the density $\rho(x)= const\times e^{-|x|}$.
As shown in the text, $T$ leaves lines through the origin
invariant, e.g. $T(x,0,0)=(\tilde{g}(x),0,0)$ for all $x$,
and the figure shows the function $\tilde{g}$.}
\end{minipage}
\hfill
\begin{minipage}[t]{0.4\textwidth}
\resizebox{!}{4cm}{\includegraphics{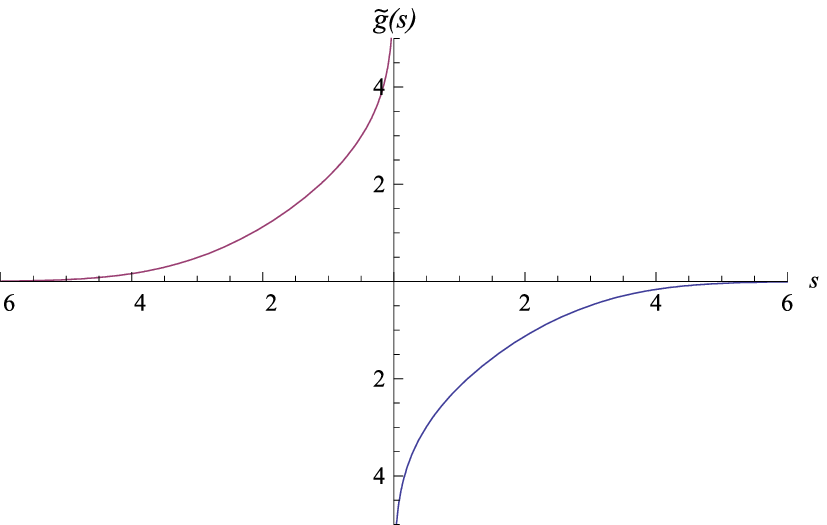}}
\caption{\label{jAs2}  Optimal transport map $T$ for the density $\rho(x)= const\times e^{-|x|^2/2}$.}
\end{minipage}
\hfill~
\end{figure}

\bexam (normal radially symmetric distribution)
\label{Example2}
Assume now that $\rho(x)=\frac{1}{Z}e^{-|x|^2/2}$ for $x\in\R^3$, where $Z$ is the normalizing constant. Then for $t\in (0,\infty)$, we have
$$ F_1(t)=\frac{1}{\sqrt{2\pi}}\int_0^te^{-s^2/2}s^{2}ds,~F_2(-t)=\frac{1}{\sqrt{2\pi}}\int_{t}^{\infty}e^{-s^2/2}s^{2}ds~\mbox{and}~g(t)=F^{-1}_2(F_1(t)).$$
\eexam

\section{Asymptotic exactness of the optimal transport functional in the semiclassical limit}

Our goal in this section is to compare the \textit{exact} quantum mechanical ground state energy to the
\textit{approximate} DFT ground state energy obtained by replacing $V_{ee}$ by the optimal
transportation functional. Recall from (\ref{EQM}) and (\ref{VP})) that the exact ground state energy
of an $N$-electron system is defined as
\begin{equation}
\label{exact}
   E_0^{QM}=\inf_{\Psi\in {\cal A}}\Big\{\,T[\Psi]+V_{ne}[\rho^\Psi]+V_{ee}[\rho_2^\Psi]\,\Big\}
\end{equation}
and the approximate ground state energy is (recall the DFT formalism in \eqref{DFTen}, \eqref{DFTen2})
\begin{equation}
\label{approx}
   E_0^{DFT-OT}=\inf_{\Psi\in {\cal A}} \Big\{\,T[\Psi]+V_{ne}[\rho^\Psi]+E_{OT}[\rho^\Psi]\,\Big\} = \inf_{\rho\in\calR} \Big\{ T_{QM}[\rho] + V_{ne}[\rho] + 
   E_{OT}[\rho]\Big\}.
\end{equation}
In the above, $\rho_2^\Psi$ and $\rho^\Psi$ denote the pair density, respectively the single particle density of $\Psi$ (see (\ref{rhoN}),
(\ref{rho2}) and (\ref{rho})), and $E_{OT}$ is the optimal transportation functional with Coulomb cost from \eqref{EMT}
\begin{equation}
   \label{EOT}
   E_{OT}[\rho]=\inf_{\gamma\in \Gamma(\rho,\rho)}\int_{\R^6}\frac{1}{|x-y|}d\gamma(x,y).
\end{equation}
Due to the fact that $\rho^\Psi$ is the marginal of $\rho_2^\Psi$, we have
\begin{equation}\label{Veelowbd}
   V_{ee}[\rho_2^\Psi]\ge E_{OT}[\rho^\Psi] \mbox{ for every }\Psi\in {\cal A}.
\end{equation}
Taking the infimum over $\Psi$ gives
\begin{theorem} \label{lowbound} For every $N$, and any potential $v\in L^{3/2}(\R^3)+L^\infty(\R^3)$, the density functional with electron-electron interaction energy
given by the mass transportation functional is a rigorous lower bound:
   $$E_0^{QM}\ge E_0^{DFT-OT}.$$
\end{theorem}
Now consider the kinetic energy functional $T[\Psi]$ from Section 2.1 with physical constants inserted,
\beq \label{kinen}
  T_{\hbar}[\Psi] = \frac{{\hbar}^2}{2m} \int ... \int \sum_{i=1}^N |\nabla_{x_i}\Psi(x_1,s_1,..,x_N,s_N)|^2
  dz_1..dz_N.
\eeq
Here $m$ is the mass of the electron and $\hbar$ is Planck's constant $h$ divided by $2\pi$.
We are interested in the
limit $\hbar\rightarrow 0$ (\textit{semiclassical limit}). Define now $E_0^{QM}(\bar{h})$ and $E_0^{DFT-OT}(\bar{h})$ as in (\ref{exact}) and (\ref{approx}),
but with $T[\Psi]$ replaced by $T_{\hbar}[\Psi]$. Note now that the statement
\be \label{quotientlimit}
  \frac{E_0^{QM}(\hbar)}{E_0^{DFT-OT}(\hbar)}\rightarrow 1~\mbox{as}~\hbar\rightarrow 0
\ee
is in general false.
The reason is that when $\hbar$ gets small, then (for typical $V_{ne}$)
the ground state densities of both models contract, and
the approximation is not uniformly good on families of contracting densities.

This has nothing particular to do with the use of $E_{OT}$, but (\ref{quotientlimit}) fails for {\it any} DFT model (\ref{DFTen}) whose electron interaction
functional $\tilde{V}_{ee}$ has the correct scaling under dilations,
\beq \label{dilbeh}
   \tilde{V}_{ee}[\alpha^3\rho(\alpha\cdot)]
   =\alpha \tilde{V}_{ee}[\rho(\cdot)],
\eeq
such as the mean field functional (\ref{indansatz}) or the local density
approximation (\ref{LDA}). A counterexample is already given
by atoms, $v(x)=-Z/|x|$ (eq. (\ref{potential}) in Section 2, with $\alpha=1$). Very
remarkably, in this case
\begin{equation}
\label{star}
   E_0^{QM}(\hbar) = \frac{E_0^{QM}(1)}{\hbar^{ 2}} ~\mbox{and}~E_0^{DFT-OT}(\hbar) = \frac{ E_0^{DFT-OT}(1)}{\hbar^{ 2}},
\end{equation}
and hence the quotient $E_0^{QM}(\hbar)/E_0^{DFT-OT}(\hbar)$ is independent of $\hbar$!
To prove this, use that the four functionals involved, $T, V_{ne}, V_{ee}$, and
$\tilde{V}_{ee}$, all have a definite scaling behaviour with respect to dilations.
For a given $\Psi\in {\mathcal A}$, consider its $L^2$-norm-preserving dilation
$$\Psi_{\hbar}(x_1,..,x_N) := (\hbar^2)^{-3N/2} \Psi(\hbar^{-2} x_1,..,\hbar^{-2} x_N).$$
Then (with $T_{\hbar}$ being the kinetic energy with prefactor $\hbar^2/2m$ from (\ref{kinen}))
$$(T_{\hbar} + V_{ne} + V_{ee})[\Psi_{\hbar}] = \hbar^{-2} (T_1 + V_{ne} + V_{ee})[\Psi].$$
Taking the infimum over $\Psi$ gives the first assertion in (\ref{star}). The second assertion follows analogously after noting that
$$
               \rho^{\Psi_{\hbar}}(x) = (\hbar^2)^{-3} \rho^{\Psi}(\hbar^{-2}x), \;\; E_{OT}[\rho^{\Psi_\hbar}] = \hbar^{-2} E_{OT}[\rho^\Psi]
$$
(or more generally $\tilde{V}_{ee}[\rho^{\Psi_\hbar}]= \hbar^{-2} \tilde{V}_{ee}[\rho^{\Psi_\hbar}]$ for every $\tilde{V}_{ee}$ satisfying
(\ref{dilbeh})).

What we can prove is the following ``pointwise'' statement in which
we only minimize out $\Psi$ at fixed $\rho$:
\begin{theorem}
\label{limit}
Let $N=2$. Then
$$\lim_{\hbar\to 0}  F_{HK}[\rho] = E_{OT}[\rho]~\mbox{for every}~\rho\in {\cal R}.$$
Here ${\cal R}$ is the natural class of densities given by the image of $\cala$ under
the map $\Psi \mapsto \rho$ (${\cal R}$ is defined in (\ref{admdens2})), and
$F_{HK}$ is the Hohenberg-Kohn functional (\ref{FHK}) with kinetic energy functional $T_\hbar$ in place
of $T$.
\end{theorem}
This case for $N=2$ already
contains the main analytic issue, namely that the optimal transport
measure $\gamma$ is singular and so its square root fails to be in $L^2$ and
fails to
have an $L^2$ gradient. But the case allows to avoid the quantum mechanical
issues of spin and antisymmetry, which would enter on top of this
when $N\ge 3$.

An interesting challenge raised by the above theorem is to derive higher order corrections to $E_{OT}$
in the semiclassical limit.

\subsection{Re-instating the constraint} \label{S:Constraint}

 In order to show that $\lim_{\hbar\to 0}  F_{HK}[\rho] = E_{OT}[\rho]$, we will need to make modifications to the optimal
plan $\gamma$ which yields $E_{OT}[\rho]$, since any $\Psi$ which represents $\gamma$ has $T[\Psi]=+\infty$. Therefore
we cannot use these $\Psi'$s as
trial states in the variational principle for $F_{HK}[\rho]$. Hence, we will need to modify the optimal
$\gamma$. {\it But the modifications that one would
like to use, e.g. smoothing, lead to modified marginals.}

Hence we need to be able to control the change in $E_{OT}[\rho]$ induced by a small change in $\rho$. This is not trivial,
due to the rigid infinite-dimensional constraint in the variational principle for $E_{OT}$ that the trial states must have marginals
exactly equal to $\rho$, and is achieved in Theorem \ref{T:cont} below.

The main technical idea behind this theorem is the following construction to ``re-instate the constraint'', i.e. to deform
a given trial plan into a nearby one with prescribed marginals.
Suppose we are given an arbitrary transport plan
$\gamma_{A\rightarrow A}$ with equal marginals $\rho_A$, and an arbitrary second density $\rho_B$. We assume that
$\rho_A,\rho_B\in L^1\cap L^3(\R^3)$, $\rho_A,\rho_B\ge 0$, and $\int_{\R^3}\rho_A(x)\,dx=\int_{\R^3}\rho_B(x)\,dx=1$.
Our interest is in the case when $\rho_B$ is near $\rho_A$, but the construction works for general $\rho_B$.

Intuitively, the plan $\gamma_{B\rightarrow B}$ with equal marginals $\rho_B$ we have in mind is the following.
\begin{itemize}
\item  First transport $\rho_B$ to $\rho_A$ by a transport plan $\gamma_{B\rightarrow A}$ that does not move much mass around when $\rho_B$ is close to $\rho_A$.
\item  Then apply the plan $\gamma_{A\rightarrow A}$.
\item  Finally transport $\rho_A$ back to $\rho_B$.
\end{itemize}
First, let us construct a suitable plan $\gamma_{B\rightarrow A}$.
Let $f(x):=\min\{\rho_A(x),\rho_B(x)\}$. Take $f_A:=(\rho_A-f)_+$ and $f_B:=(\rho_B-f)_+$. Then $\rho_A=f+f_A$ and $\rho_B=f+f_B$.

On $f$ we ``do nothing'', i.e. we let:
$$\gamma_{f\rightarrow f}(x,y)=f(x)\delta_x(y).$$
On $f_A$ we transport to $f_B$ via a convenient plan which allows simple estimates (note that $\int_{\R^3} f_A(x)\,dx=\int_{\R^3} f_B(x)\,dx$, due
to the fact that $\int_{\R^3}\rho_A(x)\,dx=\int_{\R^3}\rho_B(x)\,dx$):
$${\gamma}_{f_A\rightarrow f_B}(x,y)=\frac{f_A(x)f_B(y)}{\int_{\R^3} f_B(x)\,dx}.$$
We then set
\begin{equation}
\label{gamma12}
\gamma_{A\rightarrow B}(x,y)=\gamma_{f\rightarrow f}(x,y)+{\gamma}_{f_A\rightarrow f_B}(x,y)=f(x)\delta_x(y)+\frac{f_A(x)f_B(y)}{\int_{\R^3} f_B(y)\, dy}.
\end{equation}
Note that $\int_{\R^3}\gamma_{A\rightarrow B}(x,y)\,dy=f(x)+f_A(x)=\rho_A(x)$ and
$\int\gamma_{A\rightarrow B}(x,y)\,dx=f(y)+f_B(y)\,\frac{\int_{\R^3} f_A(x)\,dx}{\int_{\R^3} f_B(x)\,dx}=f(y)+f_B(y)=\rho_B(y)$, as required.
We will also need the reverse plan
\begin{equation}
\label{gamma21}
\gamma_{B\rightarrow A}(x,y)=f(x)\delta_x(y)+\frac{f_B(x)f_A(y)}{\int_{\R^3} f_A(y)\, dy},
\end{equation}
which satisfies $\int_{\R^3}\gamma_{B\rightarrow A}(x,y)\,dy=\rho_B(x)$ and $\int_{\R^3}\gamma_{B\rightarrow A}(x,y)\,dx=\rho_A(y)$.
Finally we introduce the combined plan
\begin{equation}
\label{combplan}
P(x,w):=\int_{\R^3}\int_{\R^3}\gamma_{B\rightarrow A}(x,y)\,\frac{\chi_{\rho_A>0}(y)}{\rho_A(y)}\,\gamma_{A\rightarrow A}^{opt}(y,z)\frac{\chi_{\rho_A>0}(z)}{\rho_A(z)}\,\gamma_{A\rightarrow B}(z,w)\,dy\,dz.
\end{equation}

We now claim that
\beq \label{gamma22}
  \int_{\R^3}P(x,w)\,dw=\rho_B(x) \mbox{ and }\int_{\R^3}P(x,w)\,dx=\rho_B(w).
\eeq
To prove the first claim, we begin by integrating over $w$. This yields
$$
  \int_{\R^3}P(x,w)\,dw=\int_{\R^3}\int_{\R^3}\gamma_{B\rightarrow A}(x,y)\,\frac{\chi_{\rho_A>0}(y)}{\rho_A(y)}\,
  \gamma_{A\rightarrow A}^{opt}(y,z)\frac{\chi_{\rho_A>0}(z)}{\rho_A(z)}\,\rho_A(z)\,dy\,dz.
$$
Noting that $\frac{\chi_{\rho_A>0}(z)}{\rho_A(z)}\,\rho_A(z)=1$ whenever $\gamma_{A\rightarrow A}^{opt}(y,z)>0$ and recalling that
$\int_{\R^3}\gamma_{A\rightarrow A}^{opt}(y,z)\,dz=\rho_A(y)$, integrating over $z$ yields
$$
    \int_{\R^3}P(x,w)\,dw=\int_{\R^3}\gamma_{B\rightarrow A}(x,y)\,\frac{\chi_{\rho_A>0}(y)}{\rho_A(y)}\,\rho_A(y)\,dy.
$$
Since $\frac{\chi_{\rho_A>0}(y)}{\rho_A(y)}\,\rho_A(y)=1$ whenever $\gamma_{B\rightarrow A}(x,y)>0$ and
since $\int_{\R^3}\gamma_{B\rightarrow A}(x,y)\,dy=\rho_B(x)$, the right hand side becomes equal to $\rho_B(x)$ after integrating over $y$.
The second marginal condition can be derived analogously.

\subsection{Continuity of the optimal transport functional} \label{S:Cont}

By combining the techique introduced above with appropriate estimates,
we are able to control the change in $E_{OT}[\rho]$ induced by a small change in $\rho$.

\begin{theorem} \label{T:cont}
There exists a $c_*>0$ such that for any $\rho_A,\rho_B\in L^1\cap L^3(\R^3)$, with $\rho_A,\rho_B\ge 0$ and
$\int_{\R^3}\rho_A(x)\,dx=\int_{\R^3}\rho_A(x)\,dx=1$, the optimal transport functional with Coulomb cost (\ref{EOT}) satisfies
$$
      \left|E_{OT}[\rho_A]-E_{OT}[\rho_B]\right|\le c_*\left(||\rho_A||_{L^1\cap L^3(\R^3)}+||\rho_B||_{L^1\cap L^3(\R^3)}\right)
      ||\rho_A-\rho_B||_{L^1(\R^3)\cap L^3(\R^3)},
$$
where $||\rho_i||_{L^1\cap L^3(\R^3)}:=\max\{||\rho_i||_{L^1(\R^3)},||\rho_i||_{L^3(\R^3)}\}$ for $i\in\{A,B\}$.
\end{theorem}
\begin{proof}
Fix arbitrarily two marginals $\rho_A,\rho_B\in L^1\cap L^3(\R^3)$,
with $\rho_A,\rho_B\ge 0$ and $\int_{\R^3}\rho_A(x)\,dx=\int_{\R^3}\rho_A(x)\,dx=1$.
Let $\gamma_{A\rightarrow A}=\gamma_{A\rightarrow A}^{opt}$ be an optimal transport plan of $C$ subject to the constraint that
$\gamma_{A\rightarrow A}$ has equal marginals $\rho_A$. The main idea is to consider the associated plan $\gamma_{B\rightarrow B}=P$ introduced
in (\ref{combplan}) and show that
\begin{equation} \label{varbd}
     C(\gamma_{B\rightarrow B})\le C(\gamma_{A\rightarrow A}^{opt})+c_*\left(||\rho_A||_{L^1\cap L^3(\R^3)}+||\rho_B||_{L^1\cap L^3(\R^3)}\right)
     ||\rho_A-\rho_B||_{L^1(\R^3)\cap L^3(\R^3)}.
\end{equation}
By the variational principle for $E_{OT}[\rho_B]$ and the optimality of $\gamma_{A\rightarrow A}^{opt}$ this implies
$$
     E_{OT}[\rho_B]\le E_{OT}[\rho_A] + c_*\left(||\rho_A||_{L^1\cap L^3(\R^3)}+||\rho_B||_{L^1\cap L^3(\R^3)}\right)
      ||\rho_A-\rho_B||_{L^1(\R^3)\cap L^3(\R^3)},
$$
as required.

\textbf{Step 1:} $C(P)\le C(\gamma_{A\rightarrow A}^{opt})+ 3M$ with $M=\sup_{y\in\R^3}\int_{\R^3}c(y,w)f_B(w)\,dw$:

By substituting the expressions (\ref{gamma12}) and (\ref{gamma21}) into (\ref{combplan}), we get
\begin{multline}
\label{upsest}
C(P)=\int_{\R^3}\int_{\R^3}\int_{\R^3}\int_{\R^3}c(x,w)\left[f(x)\delta_x(y)+\frac{f_B(x)f_A(y)}{\int_{\R^3} f_A(y)\, dy}\right] \,\frac{\chi_{\rho_A>0}(y)}{\rho_A(y)}\,\gamma_{A\rightarrow A}^{opt}(y,z)\frac{\chi_{\rho_A>0}(z)}{\rho_A(z)}\\
\left[f(z)\delta_z(w)+\frac{f_A(z)f_B(w)}{\int_{\R^3} f_B(w)\, dw}\right]\,dx\,dy\,dz\,dw.
\end{multline}
This is a sum of four terms, which arise by picking one term from each square bracket and carrying out the integrals over the delta functions:
$$
  W_1=\int_{\R^3}\int_{\R^3}c(y,z)f(y)\,\frac{\chi_{\rho_A>0}(y)}{\rho_A(y)}\,
  \gamma_{A\rightarrow A}^{opt}(y,z)\frac{\chi_{\rho_A>0}(z)}{\rho_A(z)}f(z)\,dy\,dz,
$$
$$W_2=\int_{\R^3}\int_{\R^3}\int_{R^3}c(y,w)f(y)\,\frac{\chi_{\rho_A>0}(y)}{\rho_A(y)}\,\gamma_{A\rightarrow A}^{opt}(y,z)\frac{\chi_{\rho_A>0}(z)}{\rho_A(z)}\frac{f_A(z)f_B(w)}{\int_{\R^3} f_B(w)\, dw}\,dy\,dz\,dw,$$
$$W_3=\int_{\R^3}\int_{\R^3}\int_{\R^3}c(x,z)\frac{f_B(x)f_A(y)}{\int_{\R^3} f_A(y)\, dy}\,\frac{\chi_{\rho_A>0}(y)}{\rho_A(y)}\,\gamma_{A\rightarrow A}^{opt}(y,z)\frac{\chi_{\rho_A>0}(z)}{\rho_A(z)}f(z)\,dx\,dy\,dz$$
and
$$
  W_4=\int_{\R^3}\int_{\R^3}\int_{\R^3}\int_{\R^3}c(x,w)\frac{f_B(x)f_A(y)}{\int_{\R^3} f_A(y)\, dy}\,\frac{\chi_{\rho_A>0}(y)}{\rho_A(y)}\,\gamma_{A\rightarrow A}^{opt}(y,z)\frac{\chi_{\rho_A>0}(z)}{\rho_A(z)}\\
  \frac{f_A(z)f_B(w)}{\int_{\R^3} f_B(w)\, dw}\,dx\,dy\,dz\,dw.
$$
Next, we will estimate each of these four terms. For the first term, we use the simple estimate that $f\chi_{\rho_A>0}\le\rho_A$, which gives
$$W_1\le\int_{\R^3}\int_{\R^3}c(y,z)\,\gamma_{A\rightarrow A}^{opt}(y,z)\,dy\,dz=C(\gamma_{A\rightarrow A}^{opt}).$$
For the second term, we estimate $\int_{\R^3} c(y,w)\,dw$ by the constant $M$ defined in Step 2 and
$f(y)\,\frac{\chi_{\rho_A>0}(y)}{\rho_A(y)}$ by $1$, and we get
\begin{eqnarray*}
  W_2 &\le & \frac{M}{\int_{\R^3} f_B(w)\,dw}\int_{\R^3}\int_{\R^3}\gamma_{A\rightarrow A}^{opt}(y,z)\frac{\chi_{\rho_A>0}(z)}{\rho_A(z)}f_A(z)\,dy\,dz
        \\
      &  = & \frac{M}{\int_{\R^3} f_B(w)\,dw}\int_{\R^3}\chi_{\rho_A>0}(z)f_A(z)\,dz = M.
\end{eqnarray*}
Analogously, by the change of variables $(y,z,w)\mapsto (z,y,x)$ we have
$$W_3=W_2\le M.$$
Finally, to bound $W_4$ we estimate $\int_{\R^3} c(x,w)f_B(w)\,dw$ by $M$ and
$f_A(y)\,\frac{\chi_{\rho_A>0}(y)}{\rho_A(y)}$ by $1$, and we obtain
\begin{eqnarray*}
W_4&\le& \frac{M}{\left(\int_{\R^3}f_B(w)\,dw\right)^2}\int_{\R^3}\int_{\R^3}\int_{\R^3}f_B(x)\,\gamma_{A\rightarrow A}^{opt}(y,z)f_A(z)\frac{\chi_{\rho_A>0}(z)}{\rho_A(z)} \,dx\,dy\,dz\\
&\le &\frac{M}{\left(\int_{\R^3}f_B(w)\,dw\right)^2}\left[\int_{\R^3}f_B(x)\,dx\right]
   \left[\int_{\R^3}\int_{\R^3}\gamma_{A\rightarrow A}^{opt}(y,z)\,dy f_A(z)\frac{\chi_{\rho_A>0}(z)}{\rho_A(z)}\,dz\right] = M.
\end{eqnarray*}
Plugging the above bounds for $W_1, W_2, W_3,W_4$ into (\ref{upsest}) yields the assertion.

\textbf{Step 2.} For $g\in L^1\cap L^2(\R^3)$, we have:
\be \label{ineq}
   \sup_{x\in\R^3}\left|\int_{\R^3}\frac{1}{|x-y|}g(y)\,dy\right|\le c_0 \max\{||g||_{L^1(\R^3)},||g||_{L^3(\R^3)}\},
\ee
with $c_0=2\left(\frac{8\pi}{3}\right)^{1/3}$.

To prove this, we split $\frac{1}{|x-y|}$ into a short-range and a long-range part,
$$\frac{1}{|z|}=\frac{\chi_{|z|<a}}{|z|}+\frac{\chi_{|z|\ge a}}{|z|}=:h_s(z)+h_l(z),$$
with the obvious definitions for $h_s$ and $h_l$, and with cut-off parameter $a>0$ to be chosen later. Note that $h_s\in L^{3/2}(\R^3)$ and $h_l\in L^\infty(\R^3)$. By H\"older's inequality we have
\begin{multline*}
\left|\int_{\R^3}\frac{1}{|x-y|}g(y)\,dy\right|=\left|\int_{\R^3}h_s(x-y)g(y)\,dy+\int_{\R^3}h_l(x-y)g(y)\,dy \right|\le ||h_s||_{L^{3/2}(\R^3)}||g||_{L^3(\R^3)}\\
+||h_l||_{L^{\infty}(\R^3)}||g||_{L^1(\R^3)}\le ||g||_{L^1\cap L^3(\R^3)}\left(||h_s||_{L^{3/2}(\R^3)}
+||h_l||_{L^{\infty}(\R^3)}\right).
\end{multline*}
Explicitly,
$$||h_s||_{L^{3/2}(\R^3)}+||h_l||_{L^{\infty}(\R^3)}=\left(4\pi\int_0^ar^2\frac{1}{r^{3/2}}\,dr\right)^{2/3}+\frac{1}{a}=\left(\frac{8\pi}{3}\right)^{2/3}a+\frac{1}{a}.$$
Minimizing over $a$ in the above gives $a=\left(\frac{8\pi}{3}\right)^{-1/3}$, leading to the value of $c_0$ in the assertion.

\textbf{Step 3.} Putting it all together:

By Steps $1$ and $2$ we have
$$C(P)\le C(\gamma_{A\rightarrow A}^{opt})+3c_0||f_B||_{L^1\cap L^3(\R^3)}.$$
But $0\le f_B\le |\rho_A-\rho_B|$, so $||f_B||_{L^1\cap L^3(\R^3)}\le ||\rho_A-\rho_B||_{L^1\cap L^3(\R^3)}$.
This establishes (\ref{varbd}) and Theorem \ref{T:cont}, with $c_*=3c_0 = 6 (\frac{8\pi}{3})^{1/3}$.
\end{proof}

\subsection{Finiteness of kinetic energy}

In this section we investigate the behaviour of derivatives of the combined plan
$\gamma_{B\rightarrow B}=P$ introduced in (\ref{combplan}) when the original
plan $\gamma_{A\rightarrow A}$ is differentiable.

Recall that $\gamma_{A\rightarrow A}$ is a transport plan of $C$ subject to the
constraint of equal marginals $\rho_A$, $\gamma_{A\rightarrow B}$ was defined
in (\ref{gamma12}), and $\gamma_{B\rightarrow A}$ is the reverse plan (\ref{gamma21}). Unlike in
the previous section, here $\gamma_{A\rightarrow A}$ does not need to be optimal. Due to the
fact that wave functions correspond, up to integrating out variables, to square roots of pair densities,
and the kinetic energy of a wave function is
$\frac{h^2}{2}\int|\nabla\Psi|^2$, we have to show that $\nabla\sqrt{\gamma_{B\rightarrow B}}\in L^2$,
in order to be able to construct an admissible trial function with pair density $\gamma$
in the variational definition of the Hohenberg-Kohn
density functional $F_{HK}$. The following result gives hypotheses under which this is true.
Before stating the result we introduce the following notion which we call {\it strong positivity}.

\begin{defn} \label{D:pos} A transportation plan $\gamma\in\calp(\R^{2d})$ with marginals $\mu$, $\nu\in\calp(\R^d)$
is called {\it strongly positive} if there exists a constant $\beta>0$ such that
$$
    \gamma \ge \beta \mu\otimes\nu.
$$
\end{defn}
We note that strong positivity implies, in particular, that the support of $\gamma$ is the product of the
supports of its marginals, $supp~\gamma=supp~\mu\times supp~\nu$.
\begin{theorem}
\label{T:grad}
Suppose that $\rho_A,\rho_B\ge 0,\sqrt{\rho_A},\sqrt{\rho_B}\in H^1(\R^3)$, and assume that
$\gamma_{A\rightarrow A}$ belongs to the set $\calM_+(\R^6)$ (see Section 3) and has equal marginals $\rho_A$.
\begin{enumerate}
\item [(i)] $\sqrt{\gamma_{A\rightarrow A}^{opt}}\in H^1(\R^3)$~~~~(smoothness);
\item [(ii)] $\gamma_{A\rightarrow A}^{opt}\ge\beta\rho_A\otimes\rho_A$ for some constant $\beta>0$ (strong positivity).
\end{enumerate}
Then the plan $\gamma_{B\rightarrow B}=P$ defined in (\ref{combplan}) satisfies $\sqrt{P}\in H^1(\R^6)$.
\end{theorem}

\begin{proof}
Plugging formula (\ref{gamma21}) for $\gamma_{B\rightarrow A}$ into (\ref{combplan})
and using that $\rho_A,\rho_B\ge 0$ and that $\int_{\R^3} f_A(y)\,dy=\int_{\R^3} f_B(y)\,dy$, we have
$$
P(x,w)
=\frac{f(x)}{\rho_A(x)}\int_{\R^3}\gamma_{A\rightarrow A}^{opt}(x,z))\frac{\gamma_{A\rightarrow B}(z,w)}{\rho_A(z)}
\,dz +\frac{f_B(x)}{\int_{\R^3} f_B(y)\, dy}
\int_{\R^3}\int_{R^3}f_A(y)\frac{\gamma_{A\rightarrow A}^{opt}(y,z)}{\rho_A(y)\rho_A(z)}\,
\gamma_{A\rightarrow B}(z,w)\,dy\,dz.
$$
Consequently,
\begin{eqnarray*}
\lefteqn{\nabla\sqrt{P(x,w)}= \frac{1}{2\sqrt{P(x,w)}}}\\
&& \bigg[\frac{\nabla
f(x)}{\rho_A(x)}\int_{\R^3}\gamma_{A\rightarrow A}^{opt}(x,z))\frac{\gamma_{A\rightarrow B}(z,w)}{\rho_A(z)}\,dz
-\frac{f(x)}{\rho_A(x)}\frac{\nabla\rho_A(x)}{\rho_A(x)}\int_{\R^3}\gamma_{A\rightarrow A}^{opt}(x,z))
\frac{\gamma_{A\rightarrow B}(z,w)}{\rho_A(z)}\,dz
 \\
&&+\frac{f(x)}{\rho_A(x)}\int_{\R^3}\nabla_x\gamma_{A\rightarrow A}^{opt}(x,z))
\frac{\gamma_{A\rightarrow B}(z,w)}{\rho_A(z)}\,dz +
\frac{\nabla f_B(x)}{\int_{\R^3} f_B(y)\, dy}
\int_{\R^3}\int_{R^3}f_A(y)\frac{\gamma_{A\rightarrow A}^{opt}(y,z)}{\rho_A(y)\rho_A(z)}\,
\gamma_{A\rightarrow B}(z,w)\,dy\,dz\bigg]\\
&=&:\tilde{W}_1+\tilde{W}_2+\tilde{W}_3+\tilde{W}_4,
\end{eqnarray*}
with the obvious definitions for $\tilde{W}_1,\tilde{W}_2,\tilde{W}_3$ and $\tilde{W}_4$. We have to show that
$\int_{\R^3}|\tilde{W}_i(\cdot,w)|^2\,dw\in L^1(\R^3)$ for $i=1,\ldots,4$. To estimate the first two terms, we use the following lower bound on $P$ which neglects the contribution from $f_B(x)$ in $P(x,w)$.
\begin{equation}
\label{lowbd1}
P(x,w)\ge\frac{f(x)}{\rho_A(x)}\int_{\R^3}\gamma_{A\rightarrow A}^{opt}(x,z))\frac{\gamma_{A\rightarrow B}(z,w)}{\rho_A(z)}\,dz=:\frac{f(x)}{\rho_A(x)}g(x,w).
\end{equation}
It follows that
$$|{\tilde{W}}_1(x,w)|\le\frac{1}{2}\sqrt{\frac{\rho_A(x)}{f(x)g(x,w)}}\,\,\frac{\nabla f(x)}{\rho_A(x)} g(x,w)~~\mbox{and}~~|{\tilde{W}}_2(x,w)|\le\frac{1}{2}\sqrt{\frac{\rho_A(x)}{f(x)g(x,w)}} \left(-\frac{f(x)}{\rho_A(x)}\frac{\nabla\rho_A(x)}{\rho_A(x)}\right) g(x,w)$$
and hence
$$|{\tilde{W}}_1(x,w)|^2\le\frac{1}{4}\frac{|\nabla f(x)|^2}{f(x)\rho_A(x)} g(x,w)~~~\mbox{and}~~~|{\tilde{W}}_2(x,w)|^2\le\frac{1}{4}\frac{|\nabla\rho_A(x)|^2f(x)}{\rho_A(x)^3} g(x,w).$$
Next, due to $\int_{R^3}\gamma_{A\rightarrow B}(z,w)dw=\rho_A(z)$, we have that
$$\int_{\R^3} g(x,w)\,dw=\int_{\R^3}\frac{\gamma_{A\rightarrow A}^{opt}(x,z)}{\rho_A(z)}\int_ {\R^3}\gamma_{A\rightarrow B}(z,w)\,dw= \int_{\R^3}\gamma_{A\rightarrow A}^{opt}(x,z)\,dz=\rho_A(x).$$
Consequently, using the fact that $|\nabla\sqrt{a}|^2=\frac{1}{4}\frac{|\nabla a|^2}{a}$ for any function $a$, we have
\begin{equation}
\label{T1tilde}
\int_{\R^3}|{\tilde{W}}_1(x,w)|dw\le\frac{1}{4}\frac{|\nabla f(x)|}{f(x)}=|\nabla\sqrt{f}|^2~~~\mbox{and}~~~\int_{\R^3}|{\tilde{W}}_2(x,w)|dw\le\frac{1}{4}\frac{f(x)}{\rho_A(x)}\frac{|\nabla\rho_A(x)|^2}{\rho_A(x)}\le |\nabla\sqrt{\rho_A}|^2,
\end{equation}
where in the last inequality we have used $f\le\rho_A$. Since $\sqrt{f}=\min\{\sqrt{\rho_A},\sqrt{\rho_B}\}$,
and $\sqrt{\rho_A},\sqrt{\rho_B}\in H^1(\R^3)$, by a standard fact concerning Sobolev functions we have $\sqrt{f}\in H^1(\R^3)$ and
$$\nabla\sqrt{f}=\chi_{\rho_A>\rho_B}\nabla\sqrt{\rho_B}+\chi_{\rho_A\le\rho_B}\nabla\sqrt{\rho_A}~~~\mbox{a.e.}.$$
Consequently,
\begin{equation}
\label{T2tilde}
\int_{\R^3}|{\tilde{W}}_1(x,w)dw\le |\nabla\sqrt{f}|^2\le|\nabla\sqrt{\rho_B}|^2+|\nabla\sqrt{\rho_A}|^2.
\end{equation}
Next we analyze ${\tilde{W}}_3$. To this end, we make use of the identity
$$
     \nabla_x\gamma_{A\rightarrow A}^{opt}(x,z)=\nabla_x\Bigl(\sqrt{\gamma_{A\rightarrow A}(x,z)}\Bigr)^2=2\sqrt{\gamma_{A\rightarrow A}^{opt}(x,z)}\,\,
     \nabla_x\sqrt{\gamma_{A\rightarrow A}^{opt}(x,z)}.
$$
Together with (\ref{lowbd1}) this yields
$$|{\tilde{W}}_3|\le\sqrt{\frac{\rho_A(x)}{f(x)g(x,w)}}\frac{f(x)}{\rho_A(x)}\left|\int_{\R^3}\sqrt{\gamma_{A\rightarrow A}^{opt}(x,z)}\,\,\nabla_x\sqrt{\gamma_{A\rightarrow A}^{opt}(x,z)}\frac{\gamma_{A\rightarrow B}(z,w)}{\rho_A(z)}\,dz\right|.$$
To estimate the integral over $z$ in the formula above, we write
$$\frac{\gamma_{A\rightarrow B}(z,w)}{\rho_A(z)}=\sqrt{\frac{\gamma_{A\rightarrow B}(z,w)}{\rho_A(z)}}\sqrt{\frac{\gamma_{A\rightarrow B}(z,w)}{\rho_A(z)}},$$
group one of these factors with $\sqrt{\gamma_{A\rightarrow A}^{opt}(x,z)}$ and one with $\nabla_x\sqrt{\gamma_{A\rightarrow A}^{opt}(x,z)}$, and apply the Cauchy-Schwarz inequality. This yields
$$|{\tilde{W}}_3|\le\sqrt{\frac{f(x)}{\rho_A(x)g(x,w)}}\sqrt{g(x,w)}\sqrt{\int_{\R^3}\left|\nabla_x\sqrt{\gamma_{A\rightarrow A}^{opt}(x,z)}\right|^2\frac{\gamma_{A\rightarrow B}(z,w)}{\rho_A(z)}\,dz}$$
and hence
$$|{\tilde{W}}_3|^2\le\frac{f(x)}{\rho_A(x)}\int_{\R^3}\left|\nabla_x\sqrt{\gamma_{A\rightarrow A}^{opt}(x,z)}\right|^2\frac{\gamma_{A\rightarrow B}(z,w)}{\rho_A(z)}\,dz.$$
Integrating over $w$ and using that $\int_{\R^3}\frac{\gamma_{A\rightarrow B}(z,w)}{\rho_A(z)}\,dw=1$ gives
\begin{equation}
\label{T3tilde}
\int_{\R^3}|{\tilde{W}}_3(x,w)|^2\,dw\le\frac{f(x)}{\rho_A(x)}\int_{\R^3}\left|\nabla_x\sqrt{\gamma_{A\rightarrow A}^{opt}(x,z)}\right|^2\,dz\le\int_{\R^3}\left|\nabla_x\sqrt{\gamma_{A\rightarrow A}^{opt}(x,z)}\right|^2\,dz.
\end{equation}
Finally for ${\tilde{W}}_4$ is is natural to use a different lower bound for $P$ than the one in (\ref{lowbd1}), obtained by neglecting the first instead of the second term in $P(x,w)$.
\begin{equation}
\label{lowbd2}
P(x,w)\ge\frac{f_B(x)}{\int_{\R^3} f_B(y)\, dy}\int_{\R^3}\int_{R^3}f_A(y)\frac{\gamma_{A\rightarrow A}^{opt}(y,z)}{\rho_A(y)\rho_A(z)}\,
\gamma_{A\rightarrow B}(z,w)\,dy\,dz=:\frac{f_B(x)}{\int_{\R^3} f_B(y)\, dy}\tilde{g}(w).
\end{equation}
Substituting this estimate into the definition for ${\tilde{W}}_4$ immediately gives
$$|{\tilde{W}}_4|\le\frac{1}{2}\frac{|\nabla f_B(x)|}{\sqrt{f_B(x)}}\sqrt{\frac{\tilde{g}(w)}{\int_{\R^3}f_B(y)\,dy}}$$
and, after squaring, integrating over $w$, and using $\int_{\R^3}\tilde{g}(w)\,dw=\int_{\R^3}f_B(y)\,dy$, we get
\begin{equation}
\label{T4tilde1}
\int_{\R^3}|{\tilde{W}}_4(x,w)|\,dw\le\frac{1}{4}\frac{|\nabla f_B(x)|^2}{f_B(x)}=|\nabla\sqrt{f_B}|^2.
\end{equation}
But unlike the analogous bounds on ${\tilde{W}}_1,{\tilde{W}}_2,{\tilde{W}}_3,$ this estimate is insufficient to infer ${\tilde{W}}_4\in H^1(\R^6)$
since $\sqrt{\rho_A},\sqrt{\rho_B}\in H^1(\R^3)$ do \textbf{not} imply that the function
$$\sqrt{f_B}=\sqrt{\chi_{\rho_B>\rho_A}(\rho_B-\rho_A)}$$
belongs to $H^1$. In fact, even when $\sqrt{\rho_A},\sqrt{\rho_B}$ are positive and belong to $C^{\infty}$, $\sqrt{f}$ need not be in $H^1_{loc}$.

\bexam Let $\rho_A(x)=(1-x+x^2)e^{-x^2}$ and $\rho_B(x)=(1+x+x^2)e^{-x^2}$. Because $1\pm x+x^2\ge\frac{1}{2}(1+x^2)$ is bounded away from zero,
we have $\sqrt{\rho_A},\sqrt{\rho_B}\in H^1(\R)$, but $\sqrt{\chi_{\rho_B>\rho_A}(\rho_B-\rho_A)}=\chi_{(0,\infty)}(x)\sqrt{2x}e^{-x^2/2}\notin H^1(\R)$
since $|\nabla\sqrt{\chi_{\rho_B>\rho_A}(\rho_B-\rho_A)}|^2=\chi_{(0,\infty)}(x)\left(\frac{1}{2x}-2x+2x^3\right)e^{-x^2}\notin L^1(\R)$.
\eexam

Note that this example captures the \textit{generic} behaviour of $f$ near a point where any two smooth functions $\rho_A$ and $\rho_B$ cross.
This effect is the reason why the additional assumption (ii) was made in Theorem \ref{T:grad}. This assumption, together with (\ref{lowbd1}),
yields the following alternative lower bound on $P$
\begin{equation}
\label{lowbd3}
P(x,w)\ge\frac{f(x)}{\rho_A(x)}\int_{\R^3}\beta\rho_A(x)\rho_A(z)\frac{\gamma_{A\rightarrow B}(z,w)}{\rho_A(z)}\,dz=\beta f(x)\rho_B(w).
\end{equation}
We fix a number $\delta\in (0,1)$ and we use the lower bounds (\ref{lowbd2}) or (\ref{lowbd3}), depending on whether $f_B(x)\ge\delta\rho_B(x)$ or $f_B(x)<\delta\rho_B(x)$.

\textbf{Region 1:} Assume $f_B(x)\ge\delta\rho_B(x)$. Via (\ref{lowbd2}) and (\ref{T4tilde1}) we obtain
\begin{eqnarray}
\label{T4tilde2}
\chi_{f_B\ge\delta\rho_B}\int_{\R^3}|{\tilde{W}}_4(\cdot,w)|^2 dw
&\le& \frac{1}{4\delta}\frac{|\nabla f_B|^2}{\rho_B}
\le\frac{1}{2\delta}\chi_{\rho_B>\rho_A} \frac{|\nabla\rho_A|^2+|\nabla\rho_B|^2}{\rho_B}\nonumber\\
&\le & \frac{1}{2\delta}\chi_{\rho_B>\rho_A} \left(\frac{|\nabla\rho_A|^2}{\rho_A}+\frac{|\nabla\rho_B|^2}{\rho_B}\right)
 \le \frac{2}{\delta}\left(|\nabla\sqrt{\rho_A}|^2+|\nabla\sqrt{\rho_B}|^2 \right).
\end{eqnarray}

\textbf{Region 2:} Assume $f_B(x)\le\delta\rho_B(x)$. First of all, note that whenever $f_B(x)>0$, i.e. $\rho_B(x)>\rho_A(x)$, we have the following equivalences
$$f_B(x)\le\delta\rho_B(x)\Leftrightarrow\rho_B(x)-\rho_A(x)\le\delta\rho_B(x)\Leftrightarrow\rho_B(x)(1-\delta)\le\rho_A(x)=\min\{\rho_A(x),\rho_B(x)\}=f(x).$$
Via (\ref{lowbd3}) we have
$$|{\tilde{W}}_4|\le\frac{1}{2\sqrt{\beta f(x)\rho_B(w)}}\,\,\frac{|\nabla f_B(x)|}{\int_{\R^3} f_B(y)\,dy}\,\,\tilde{g}(w). $$
We split the factor ${\tilde{g}}(w)$ into $\sqrt{{\tilde{g}}(w)} \sqrt{{\tilde{g}}(w)}$ and estimate one of the factors via the elementary
inequality $f_A(y)\le\rho_A(y),$ so as to eliminate the bad factor $\sqrt{\rho_B}$ from the denominator:
\begin{eqnarray*}
\sqrt{{\tilde{g}}(w)}&=&\left(\int_{\R^3}\int_{\R^3}\frac{f_A(y)}{\rho_A(y)}\frac{\gamma_{A\rightarrow A}^{opt}(y,z)}{\rho_A(z)}\gamma_{A\rightarrow B}(z,w)dy
dz\right)^{1/2}\le\left(\int_{\R^3}\int_{\R^3}\gamma_{A\rightarrow A}^{opt}(y,z)dy \frac{\gamma_{A\rightarrow B}(z,w)}{\rho_A(z)}dz\right)^{1/2}\\
&=&\left(\int_{\R^3}\rho_A(z)\frac{\gamma_{A\rightarrow B}(z,w)}{\rho_A(z)}dz\right)^{1/2}=\sqrt{\rho_B(w)}.
\end{eqnarray*}
Consequently,
$$|{\tilde{W}}_4|\le\frac{1}{2\sqrt{\beta f(x)}}\frac{|\nabla f_B(x)|}{\int_{\R^3} f_B(y)dy}\sqrt{{\tilde{g}}(w)}.$$
Squaring, integrating over $w$ and using $\int_{\R^3}{\tilde{g}}(w)dw=\int_{\R^3} f_B(y) dy$ yields
\begin{equation}
\label{T4tilde3}
\int_{\R^3}|{\tilde{W}}_4(x,w)|^2dw\le\frac{1}{4\beta}\frac{|\nabla f_B(x)|^2}{f(x)}\frac{1}{\int_{\R^3}f _2(y) dy}.
\end{equation}
But in the region $\{x\, | \, f_B(x)>0\}=\{x\, |\, \rho_B(x)>\rho_A(x)\}$, as shown above we have $f=\rho_A\ge (1-\delta)\rho_B$, and consequently,
\begin{eqnarray}
\label{T4tilde4}
\chi_{f_B\le\delta\rho_B}\int_{R^3}|{\tilde{W}}_4(\cdot,w)|^2dw &\le&\frac{1}{2\beta}\frac{|\nabla\rho_A|^2+|\nabla\rho_B|^2}{f}\frac{1}{\int_{\R^3}f_B(y)dy}\nonumber\\
&\le&
\frac{1}{2\beta}\left(\frac{|\nabla\rho_A|^2}{\rho_A}+\frac{1}{1-\delta}\frac{|\nabla\rho_B|^2}{\rho_B}\right)\frac{1}{\int_{\R^3}f_B(y)dy}\nonumber\\
&=&\frac{2}{\beta}\left(|\nabla\sqrt{\rho_A}|^2+\frac{1}{1-\delta}|\nabla\sqrt{\rho_B}|^2\right)\frac{1}{\int_{\R^3}f_B(y)dy}.
\end{eqnarray}
Combining (\ref{T1tilde}), (\ref{T2tilde}), (\ref{T3tilde}), (\ref{T4tilde2}) and (\ref{T4tilde4}) establishes the theorem.
\end{proof}

\begin{remark}
In region $1$, the factor $\frac{1}{f(x)}$ appearing in (\ref{T4tilde3}) is uncontrollably bad. In region $2$, the factor $\frac{1}{f_B(x)}$ appearing in (\ref{T4tilde1})
is uncontrollably bad.
\end{remark}
\subsection{Smoothing}

The third ingredient needed in the proof of Theorem \ref{limit} lies in the fact that
the Coulomb cost functional $C[\gamma]=\int |x-y|^{-1} d\gamma(x,y)$ is well behaved under smoothing of $\gamma$,
despite the fact that the cost function $|x-y|^{-1}$ is discontinuous and hence does not belong to the dual of the space of probability measures on $\R^{6}$.

Let $\rho\in{\cal R}$ (see (\ref{admdens2})), and let $\gamma\in\Gamma(\rho,\rho)$ be a minimizer of $C[\gamma]=\int |x-y|^{-1}d\gamma(x,y)$
so that
$$
   C[\gamma]=E_{OT}[\rho].
$$
We now introduce a standard mollification of $\gamma$, as follows.
Let $\phi \, : \, \R^3\to\R$ belong to the Schwartz space
${\cal S}(\R^3)$ of smooth, rapidly decaying functions, and assume that $\phi>0$, $\int_{\R^3}\phi=1$,
$\phi$ radially symmetric. E.g., the choice $\phi(x)=\pi^{-3/2}e^{-|x|^2}$ will do. Let
$$
   \phi_\epsilon(x) = \frac{1}{\epsilon^{3}}\phi(\frac{x}{\epsilon}),
$$
and let $\gamma_\epsilon = (\phi_\epsilon\otimes\phi_\epsilon)*\gamma$, that is to say
\begin{equation} \label{smooth}
   \gamma_\epsilon(x,y) = \int_{\R^6}\phi_\epsilon(x-x')\phi_\epsilon(y-y')d\gamma(x',y').
\end{equation}
\begin{proposition} \label{P:smooth} The mollified pair density $\gamma_\epsilon$ introduced in (\ref{smooth})
satisfies
\begin{itemize}
\item[{(a)}] $\gamma_\epsilon\in C^\infty(\R^6)$, $\gamma_\epsilon>0$
\item[{(b)}] $\int \gamma_\epsilon(x,y)\, dx = \rho_\epsilon(y)$, $\int \gamma_\epsilon(x,y)\, dy =
\rho_\epsilon(x)$, where $\rho_\epsilon$ is the mollified marginal $(\phi_\epsilon * \rho)(x)=\int_{\R^3}\phi_\epsilon(x-x')\rho(x')dx'$.
\item[{(c)}] $C[\gamma_\epsilon]\le C[\gamma]$.
\end{itemize}
\end{proposition}
\bproof (a): Smoothness is a standard fact concerning mollification of Radon measures, and positivity
is obvious from the positivity of $\phi$. \\[2mm]
(b): This follows from the elementary calculation
\begin{eqnarray*}
  \int\gamma_\epsilon(x,y) \, dx &=& \int \int \int \phi_\epsilon(x-x') \phi_\epsilon(y-y') d\gamma(x',y') \, dx \\
  & = & \int\int \phi_\epsilon(y-y') \, d\gamma(x',y') = \int \phi_\epsilon(y-y') \rho(y') dy'.
\end{eqnarray*}
(c): First, we claim that the cost functional evaluated at the mollified transport plan,
$C[\gamma_\epsilon]$, can be interpreted as a cost functional with modified cost function evaluated
at the original transport plan. Indeed, by Fubini's theorem
\begin{eqnarray*}
  C[\gamma_\epsilon] &=& \int\int c(x,y)\Bigl[ \int\int\phi_\epsilon(x-x')\phi_\epsilon(y-y') \,
                       d\gamma(x',y')\Bigr] \, dx \, dy \\
  &=& \int\int \Bigl[ \underbrace{\int\int\phi_\epsilon(x-x')\phi_\epsilon(y-y')c(x,y) \, dx\, dy}_{=:\tilde{c}(x',y')}\Bigr] \,
                       d\gamma(x',y').
\end{eqnarray*}
The modified cost function $\tilde{c}(x',y')$
appearing here has an interesting physical meaning which we will exploit to establish (c),
namely it is the electrostatic repulsion between the two charge distributions $\phi_\epsilon(\cdot - x')$
and $\phi_\epsilon(\cdot - y')$ (i.e., the charge distributions centered at $x'$ respectively $y'$ whose profile
is given by the mollifier $\phi_\epsilon$). Now it is a standard fact going back to Newton that the
electrostatic potential exerted by a radial charge distribution on a point outside it equals the
potential exerted by the same amount of charge placed at the centre,
$$
    \frac{1}{|S_r|}\int_{S_r} \frac{1}{|x-a|} dH^2(x) = \frac{1}{\max\{|a|,r\} },
$$
where $S_r$ denotes the sphere of radius $r$ around $0$, $H^2$ is the Hausdorff measure (area element)
on the sphere, and $|S_r|(=4\pi r^2)$ is the total area of the sphere. This together with the radial
symmetry of $\phi_\epsilon$ (i.e., $\phi_\epsilon(x)=\tilde{\phi_\epsilon}(|x|)$ for some function $\tilde{\phi_\epsilon}$)
gives
\begin{eqnarray} \label{newtonest}
   \int_{\R^3} \phi_\epsilon(x) \frac{1}{|x-a|} dx &=&
      \int_{r=0}^\infty \int_{x\in S_r} \tilde{\phi_\epsilon}(r) \frac{1}{|x-a|} dH^2(x)
     = \int_0^\infty |S_r| \, \tilde{\phi_\epsilon}(r) \frac{1}{\max\{|a|,r\} } dr \nonumber \\
    &\le & \Bigl( \int_0^\infty |S_r| \, \phi_\epsilon(r) \, dr\Bigr) \, \frac{1}{|a|} = \frac{1}{|a|}.
\end{eqnarray}
Hence by repeated application of (\ref{newtonest})
\begin{eqnarray}
  \tilde{c}(x',y') &=& \int\int \phi_\epsilon(x)\phi_\epsilon(y) \frac{1}{|x+x'-(y+y')|} dx \, dy \\
                   &\le & \int\phi_\epsilon(y)\frac{1}{|x'-(y+y')|} dy
                     =    \int \phi_\epsilon(y) \frac{1}{|y-(x'-y')|} dy \\
                   &\le & \frac{1}{|x'-y'|}.
\end{eqnarray}
This establishes (c).
\eproof
\subsection{Passage to the limit}

We are now in a position to give the
\\[2mm]
{\em Proof of Theorem~\ref{limit}}. Let $\rho$ be any density in $\calR$. Recall that
$\rho\in\calR$ implies that $\sqrt{\rho}$, $\nabla\sqrt{\rho}\in L^2(\R^3)$ and hence, by the
Sobolev embedding theorem, $\sqrt{\rho}\in L^6(\R^3)$, whence $\rho\in L^1\cap L^3(\R^3)$.

We have to show that $\lim_{\hbar\to 0}F_{HK}[\rho] = E_{OT}[\rho]$. We will do so via the following
strategy:
\begin{itemize}
\item Start from an optimal transport plan $\gamma$ with marginals $\rho$.
\item Smooth it.
\item Make it strongly positive (see Definition \ref{D:pos}), by mixing in a small amount of the mean field (i.e., tensor product) plan.
\item Re-instate the marginal constraint, via the technique introduced in Section \ref{S:Constraint}.
\item Infer from Theorem \ref{T:grad} that, unlike the original optimal transport plan
$\gamma$, the so-obtained modified plan $P$ is the pair density
(\ref{rho2})
of a wave function $\Psi$ with square-integrable gradient.
\item Pass to the semiclassical limit, by careful error estimates on the three modification
steps listed above (smoothing, achieving strong positivity, re-instating the constraint).
\end{itemize}
We now implement this strategy in detail.
Let $\gamma$ be an optimal transport plan of the Coulomb cost functional $C$ subject to equal
marginals $\rho$. (Of course we know from Section 3 that $\gamma$ is unique, but uniqueness is not needed here.)
For $\epsilon>0$, let $\gamma_\epsilon$ be its mollification (\ref{smooth}).
By Proposition \ref{P:smooth}, its right and left marginals are given by the mollification
$\rho_\epsilon=\phi_\epsilon * \rho$ of the density $\rho$. Finally
we introduce the ``strong positivization''
$$
  \gammatildeepsilonbeta := (1-\beta)\gamma_\epsilon + \beta \rho_\epsilon\otimes\rho_\epsilon,
$$
where $\beta\in(0,1)$.
Note that $\gammatildeepsilonbeta$ has the same marginals as $\gamma_\epsilon$, regardless of the value of $\beta$.

Observe now that
the transportation plan $\gammatildeepsilonbeta$ and the densities $\rho_\epsilon$, $\rho$ satisfy the assumptions
of Theorem \ref{T:grad}. Consequently, by Theorem \ref{T:grad} there exists a transportation plan
$P_{\epsilon,\beta}$ with marginals $\rho$ (i.e., with re-instated constraint)
whose square root belongs to $H^1(\R^6)$.

Now comes the only step where we use the assumption $N=2$. In this case we can achieve the (otherwise
highly nontrivial) antisymmetry condition on $\Psi$ appearing in (\ref{class}) purely by means of an antisymmetric
spin part. More precisely we define $\Psi\, : \, (\R^3\times\Z_2)^2\to\C$ by
$$
  \Psi(x,s,y,t) := \sqrt{P_{\epsilon,\beta}(x,y)}\frac{\alpha(s)\beta(t)-
                                                                   \beta(s)\alpha(t)}{\sqrt{2}},
$$
where $\alpha$, $\beta\, : \, \Z_2=\{\pm\frac12\}\to\C$ are given by $\alpha(s)=\delta_{1/2}(s)$,
$\beta(s)=\delta_{-1/2}(s)$. Then it is straightforward to check that $\Psi$ belongs to the admissible
set $\cala$ defined in (\ref{class}) and its
pair density, density, and kinetic energy are
$$
    \rho_B^{\Psi}= P_{\epsilon,\beta}, \;\;\;\; \rho^{\Psi} = \rho, \;\;\;\;
     T_\hbar[\Psi] = \frac{\hbar^2}{2m}\int_{\R^6}
     |\nabla\sqrt{P_{\epsilon,\beta}}|^2.
$$
It follows that
\begin{equation} \label{limitbd0}
   \lim_{\hbar\to 0}F_{HK}[\rho] \le \lim_{\hbar\to 0} \Bigr(T_\hbar(\Psi)+V_{ee}(\Psi)\Bigr) = V_{ee}(\Psi) = C(P_{\epsilon,\beta}).
\end{equation}
Next, (\ref{varbd}) yields
\begin{equation} \label{limitbd1}
   C[P_{\epsilon,\beta}] \le C[\gammatildeepsilonbeta] + c_*\Bigl(||\rho||_{\LcapL} + ||\rho_\epsilon||_{\LcapL}\Bigr)
   ||\rho-\rho_\epsilon||_{\LcapL}.
\end{equation}
Next, we claim that
\begin{equation}\label{limitbd2}
   C[\gammatildeepsilonbeta] = (1-\beta)C[\gamma_\epsilon] + \beta C(\rho_\epsilon\otimes\rho_\epsilon)
   \le C[\gamma_\epsilon] + c_0\beta ||\rho_\epsilon||_{L^1}||\rho_\epsilon||_{\LcapL}.
\end{equation}
This is immediate from the estimate
$$
   |C[f\otimes g]|\le c_0||f||_{L^1}||g||_{\LcapL} \mbox{ for any }f,\,g\in \LcapL,
$$
which follows by applying (\ref{ineq}), multiplying by $f$, and integrating over $x$.

Finally, we will need the following bound which was established in Proposition \ref{P:smooth}:
\begin{equation} \label{limitbd3}
    C[\gamma_\epsilon]\le C[\gamma].
\end{equation}
Combining the estimates (\ref{limitbd0})--(\ref{limitbd3}) yields
$$
  \lim_{\hbar\to 0} F_{HK}[\rho] \le C[\gamma] + \frac{c_*}{3}
  \Bigl[ \beta ||\rho_\epsilon||_{L^1}||\rho_\epsilon||_{\LcapL} +
   (||\rho||_{\LcapL} + ||\rho_\epsilon||_{\LcapL}) ||\rho-\rho_\epsilon||_{\LcapL}\Bigr].
$$
Letting $\beta$ and $\epsilon$ tend to zero and using that $\rho_\epsilon$, being the mollification
$\phi_\epsilon * \rho$ of $\rho$, tends to $\rho$ in $\LcapL$ as $\epsilon\to 0$  yields
$$
   \lim_{\hbar\to 0} F_{HK}[\rho] \le C[\gamma] = E_{OT}[\rho].
$$
The reverse inequality is immediate from (\ref{Veelowbd}) and the positivity of $T_\hbar$.
This completes the proof of Theorem~\ref{limit}.

\renewcommand{\theequation}{A.\arabic{equation}}
\renewcommand{\thetheorem}{A.\arabic{theorem}}
\renewcommand{\thesection}{A}
\section{Appendix}

\begin{lemma}[Legendre transforms on the line]\label{A3}
Let $l:\R\rightarrow\R\cup\{+\infty\}$ be lower semi-continuous and convex. Define its dual function $l^{\circ}$ by (\ref{legr}). Then $l^{\circ}$ satisfies the same hypotheses as $l$, and
\begin{itemize}
\item [(a)] $(\lambda,\xi)\in\partial_\cdot l$ if and only if $(\xi,\lambda)\in\partial_\cdot l^{\circ}$;
\item [(b)] the dual function of $l^{\circ}$ is $l$, that is $l=l^{\circ\circ}$;
\item [(c)] strict convexity of $l$ implies $l^{\circ}$ differentiable, where it is subdifferentiable;
\item[(iv)] $l^{\circ}(\xi)$ is non-increasing if and only if $l(\lambda)=\infty$ for all $\lambda>0$.
\end{itemize}
\end{lemma}

\begin{proof}
(a)-(c) follow from the corresponding statements in Theorem A.1 in \cite{GM96}.
Assertion (d) is easily proved similarly to Theorem A.3 (iv) in \cite{GM96}.
To verify the {\it only if} implication, suppose that $l(\lambda)$ is finite at some $\lambda>0$; we shall show that $l^{\circ}$ increases somewhere. Being convex, $l$ must be subdifferentiable at $\lambda$ (or some nearby point): $(\lambda,\xi)\in\partial_\cdot l$. Then (i) implies that $l^{\circ}$ is finite at $\xi$ and increasing: $l^{\circ}(\xi+\epsilon)\ge l^{\circ}(\xi)+\lambda\epsilon$ for some $\eps>0$.

To prove the converse, suppose that $l^{\circ}$ increases somewhere. Then one has $(\xi,\lambda)\in\partial_\cdot l^{\circ}$ for some $\xi\in\R$ and $\lambda>0$. Invoking once again (i) gives $(\lambda,\xi)\in\partial_\cdot l$, from which one concludes finiteness of $l(\lambda)$.
\end{proof}

For $x\in\bbr^d\setminus\{0\}$, denote by $\hat x:=x/|x|$ the unit vector in direction of $x$.

\begin{lemma}[subdifferentiability of the cost]
\label{A4'}
 Let $l:\R\to\R\cup\{+\infty\}$ be convex and non-increasing on $[0,\infty)$ and define $h(x):=l(|x|)$ on $\R^d$. Unless $h$ is a constant: $(x,y)\in\partial_\cdot h$ if and only if $(|x|,-|y|)\in\partial_\cdot l$ with $y=|y|\hat{x}$ and $x\neq 0$.
\end{lemma}

\begin{proof}
Fix $x\in\R^d\setminus\{0\}$ and suppose $l(\lambda)$ admits $\xi$ as a subderivative at $|x|:(|x|,\xi)\in\partial_\cdot l$. Since $l$ is convex and non-increasing, $\xi\le 0$, while for $\epsilon\in\R$,
\begin{equation}
\label{29'}
l(|x|+\epsilon)\ge l(|x|)+\epsilon\xi.
\end{equation}
Let
\begin{equation}
\label{32}
\epsilon:=|x+v|-|x|=\sqrt{|x|^2+2<x,v>+|v|^2}-|x| \, \le \, <\hat{x},v>+\frac{v^2}{2|x|},
\end{equation}
which inequality follows from $\sqrt{1+\lambda}\le 1+\frac{\lambda}{2}$.
Now $h(x+v)=l(|x+v|)\ge l(|x|+\epsilon)$, with $\epsilon=<v,\hat{x}>+o(|v|)$, as seen from (\ref{32}).
It follows immediately from Definition~\ref{3.2}(1) that $h$ is subdifferentiable at $x$, with $(x,\xi\hat{x})\in\partial_\cdot h$. On the other hand, $h$ cannot be subdifferentiable at the origin as $h(0)=\infty$.

Now let $(x,y)\in\partial_\cdot h$, so $x\neq 0$ and for small $v\in\R^d$
$$h(x+v)\ge h(x)+<v,y>+o(|v|).$$
Spherical symmetry of $h$ forces $y$ to be parallel to $x$: otherwise, a slight rotation $x+v:=x\cos\theta-\hat{z}\sin\theta$ of $x$ in the direction $z:=y-(<y,\hat{x}>)\hat{x}$ would contradict $h(x+v)=h(x)$ for $\theta$ sufficiently small. Moreover, taking $v:=\epsilon\hat{x}$ yields (\ref{29'}) with $\xi:=<\hat{x},y>+o(1)$, which concludes the lemma: $|y|=\pm<\hat{x},y>$ holds with a minus sign since $l$ cannot increase.
\end{proof}

\begin{lemma}[uniform subdifferentiability of the cost]
\label{A5}
Let $l$ and $h$ be defined as in the lemma above. Then $h$ is subdifferentiable on $\R^d\setminus\{0\}$. Moreover, for $\delta>0$, there is a real function $O_{\delta}(\lambda)$ tending to zero linearly with $|\lambda|$ such that $|x|>\delta, y\in\partial_\cdot h(x)$ and $v\in \R^d$ imply
\begin{equation}
\label{30}
h(x+v)\ge h(x)+<v,y>+O_{\delta}(v^2).
\end{equation}
\end{lemma}

\begin{proof}
For $\lambda>0$, the convex function $l$ admits a subgradient $\xi\in\partial_\cdot l(\lambda)$: for example, take its right derivative $\xi=l'(\lambda^+)$. If $|x|=\lambda$, the lemma implies $(x,\xi\hat{x})\in\partial_\cdot h$, so $h(x)$ is subdifferentiable at $x$.

Now suppose that $(x,y)\in\partial_\cdot h$. The opposite implication of the lemma yields  $y=´\xi\hat{x}$ with $(|x|,\xi)\in\partial_\cdot l$ so (\ref{29}) holds. Morover, $\xi\le 0$. If $v\in\R^d$, then $h(x+v)\ge l(|x|+\epsilon)$ where $\epsilon$ is as in (\ref{32}).
By convexity of $l$, its right derivative is a non-decreasing function of $\lambda$. Asssume $|x|>\delta$ so that $\xi\ge l'(\delta^+)$. Together with (\ref{29'}) and (\ref{32}), this assumption gives
$$h(x+v)\ge h(x)+<\xi\hat{x},v>+v^2l'(\delta^+)/{2\delta}.$$
\end{proof}

\begin{lemma}
\label{A6}
Let $l$ and $h$ be defined as in the Lemma \ref{A4'}. Define the dual function $h^*:\R^d\rightarrow\R\cup\{+\infty\}$ via (\ref{legr}). The for some $R\ge 0$,
\begin{enumerate}
\item [(i)] $h^*(y)$ is continuously differentiable on $|y|>R$ while $h^*=+\infty$ on $|y|<R$;
\item [(ii)] $(y,x)\in\partial_\cdot h^*$ with $x\neq 0$ if and only if $(x,y)\in\partial_\cdot h$ with $y\neq 0$;
\item [(iii)] if $(y,x)\in\partial_\cdot h^*$, then $x=\nabla h^*(y)$.
\end{enumerate}
\end{lemma}

\begin{proof}
The proof follows the same reasoning as the proof of Proposition A.6 (i)-(iii) from \cite{GM96} and will be omitted.
\end{proof}

\subsubsection*{Acknowledgements}
We thank Robert McCann for many useful explanations and suggestions regarding optimal transport techniques and literature.

\end{document}